\documentclass{amsart}
\usepackage{amsmath,amsthm,amssymb,hyperref,mathrsfs}
\usepackage{aliascnt}
\usepackage{lmodern}
\usepackage[T1]{fontenc}

\usepackage{xcolor}
\definecolor{dblue}{rgb}{0,0,0.70}
\hypersetup{
	unicode=true,
	colorlinks=true,
	citecolor=dblue,
	linkcolor=dblue,
	anchorcolor=dblue
}

\makeatletter
\expandafter\g@addto@macro\csname th@plain\endcsname{%
		\thm@notefont{\bfseries}
	}%
\expandafter\g@addto@macro\csname th@remark\endcsname{%
		\thm@headfont{\scshape}
	}%
\makeatother

\newtheorem{theorem}{Theorem}[section]	
\newtheorem*{theorem*}{Theorem}

\newaliascnt{lemma}{theorem}
\newtheorem{lemma}[lemma]{Lemma}
\aliascntresetthe{lemma}

\newaliascnt{proposition}{theorem}
\newtheorem{proposition}[proposition]{Proposition}
\aliascntresetthe{proposition}

\newaliascnt{corollary}{theorem}
\newtheorem{corollary}[corollary]{Corollary}
\aliascntresetthe{corollary}

\theoremstyle{remark}

\newtheorem*{remark}{Remark}
\newtheorem*{question}{Question}

\theoremstyle{definition}
\newtheorem{definition}[theorem]{Definition}

\newaliascnt{example}{theorem}
\newtheorem{example}[example]{Example}
\aliascntresetthe{example}

\renewcommand{\restriction}{\mathbin\upharpoonright}

\newcommand{\axiom}[1]{\mathsf{#1}} 
\newcommand{\ZFC}{\axiom{ZFC}}
\newcommand{\AC}{\axiom{AC}}
\newcommand{\DC}{\axiom{DC}}
\newcommand{\ZF}{\axiom{ZF}}

\newcommand{\Ord}{\mathrm{Ord}}
\newcommand{\HOD}{\mathrm{HOD}}

\newcommand{\IS}{\axiom{IS}}
\newcommand{\HS}{\axiom{HS}}
\newcommand{\KWP}{\axiom{KWP}}
\newcommand{\SVC}{\axiom{SVC}}

\DeclareMathOperator{\dom}{dom}

\DeclareMathOperator{\supp}{supp}
\DeclareMathOperator{\rank}{rank}
\DeclareMathOperator{\sym}{sym}
\DeclareMathOperator{\fix}{fix}

\DeclareMathOperator{\id}{id}
\DeclareMathOperator{\aut}{Aut}

\DeclareMathOperator{\Add}{Add}

\newcommand{\forces}{\mathrel{\Vdash}}
\newcommand{\nforces}{\mathrel{\not{\Vdash}}}
\newcommand{\gaut}[1]{{\textstyle\int_{#1}}}

\newcommand{\PP}{\mathbb{P}}
\newcommand{\power}{\mathcal{P}}

\newcommand{\QQ}{\mathbb{Q}}

\newcommand{\BB}{\mathbb B}
\newcommand{\cF}{\mathcal F}
\newcommand{\sF}{\mathscr F}
\newcommand{\cG}{\mathcal G}
\newcommand{\sG}{\mathscr G}

\newcommand{\tup}[1]{\langle#1\rangle}

\newcommand{\middd}{\mathrel{}\middle|\mathrel{}}

\author{Asaf Karagila}
\thanks{This paper is part of the author's Ph.D.\ written in the Hebrew University of Jerusalem under the supervision of Prof.\ Menachem Magidor.}
\thanks{This research was partially done whilst the author was a visiting fellow at the Isaac Newton Institute for Mathematical Sciences in the programme `Mathematical, Foundational and Computational Aspects of the Higher Infinite' (HIF) funded by EPSRC grant EP/K032208/1.}
\address{Einstein Institute of Mathematics.
Edmond J. Safra Campus, Givat Ram.
The Hebrew University of Jerusalem. 
Jerusalem, 91904, Israel}
\curraddr{School of Mathematics,
University of East Anglia.
Norwich, NR4~7TJ, United Kingdom
}
\email{karagila@math.huji.ac.il}
\urladdr{http://karagila.org}
\date{October 26, 2018}

\title{Iterating Symmetric Extensions}

\begin{document}
\begin{abstract}
The notion of a symmetric extension extends the usual notion of forcing by identifying a particular class of names which forms an intermediate model of $\ZF$ between the ground model and the generic extension, and often the axiom of choice fails in these models. Symmetric extensions are generally used to prove choiceless consistency results. We develop a framework for iterating symmetric extensions in order to construct new models of $\ZF$. We show how to obtain some well-known and lesser-known results using this framework. Specifically, we discuss Kinna--Wagner principles and obtain some results related to their failure.
\end{abstract}
\maketitle    
\setcounter{tocdepth}{1}
\tableofcontents
\newpage
\section{Introduction}\label{section:intro}
Forcing is one of the main tools in modern set theory. It allows us to obtain consistency results by taking well-behaved models and extending them in a very controlled way, so we can ensure the new models satisfy certain statements. While it is true that the basic mechanics of forcing do not require the axiom of choice, the axiom of choice is deeply integrated into how we use forcing. Firstly, the axiom of choice cannot be violated via forcing; and secondly, the axiom of choice is used constantly when we argue with chain conditions, the mixing lemma, closure properties and more. Consider these two folklore theorems:
\begin{theorem*}
``If $p\forces\exists x\varphi(x)$, then $\exists \dot x: p\forces\varphi(\dot x)$'' is equivalent to $\AC$.\qed
\end{theorem*}
\begin{theorem*}
For every $\kappa$, ``If $\PP$ is a $\kappa$-closed forcing, then $\PP$ does not add $\kappa$-sequences to the ground model'' is equivalent to $\DC_\kappa$.\qed
\end{theorem*}

In fact, since the axiom of choice is equivalent to the statement ``Every partial order contains a maximal antichain'', without the axiom of choice we cannot even ensure that there are maximal antichains in our forcing notion. Indeed, forcing without the axiom of choice admits quite a handicap.
\medskip

In order to prove the relative consistency of $\ZF+\lnot\AC$, something that cannot be done via forcing alone, Cohen incorporated the arguments developed and perfected by Fraenkel, Mostowski and Specker for constructing permutation models for set theory with atoms (or where the axiom of regularity fails in certain ways). This technique is now known as symmetric extensions, which is an additional structure to the technique of forcing, which allows us to identify an intermediate model where the axiom of choice fails. As the theory of forcing developed, so did our understanding of symmetric extensions. Serge Grigorieff showed in \cite{Grigorieff:1975} that symmetric extensions are exactly $\HOD(V\cup x)^{V[G]}$ for some set $x$, for example, and that $V[G]$ is a generic extension of all of its symmetric extensions.
\medskip

In this paper we develop a technique for iterating symmetric extensions. We can see precursors to this idea in the works of Gershon Sageev \cite{Sageev:1975,Sageev:1981} and Gordon Monro \cite{Monro:1973}. Sageev constructs two models using iterations and symmetries. Despite the fact that Sageev's constructions are not iterations of symmetric extensions per se, it is clear from his arguments that the solution is morally an iteration of symmetric extensions. Monro in his work constructs a sequence of models that each one is a ``morally'' a symmetric extension of the previous.

The motivation for developing such a technique is coming from two seemingly unrelated problems. The first motivation is forcing over $\ZF$ models. If we force over symmetric extensions---or models close enough to be symmetric extensions---we can try and treat the forcing as an iteration over a ground model satisfying $\ZFC$. There we can take some advantage of the axiom of choice and its consequences on forcing. So one goal is to help and simplify forcing over models of $\ZF$.

The second motivation is coming from the fact that it is sometimes easier to break a large problem into parts, and solve only a few parts of the problem at a time, over and over again, rather than the entire problem at once. This is self-evident by the ubiquity of problems solved via iterated forcing. This approach might allow for simplification of Sageev's works, which are extremely technical and are long overdue for a good revision. We hope to address these in the future.
\medskip

This paper starts by covering some preliminaries about iterated forcing and symmetric extensions.\footnote{Well, it actually begins with the introduction, motivation and structure of the paper, but we ignore that for obvious reasons here.} In sections \ref{section:ext aut} and \ref{section:ext filters} we will discuss how to extend the general structure of symmetric extensions to the iteration, and apply these in section~\ref{section:IS} to discuss the class of names and the forcing relation identifying the intermediate model. We will prove in section~\ref{section:justification} that the term ``iterated symmetric extensions'' is justified, by showing that the iteration of symmetric extensions is in fact a symmetric iteration and vice versa. Section~\ref{section:productive} will be devoted to a variant of the general method where some added restrictions will allow us to access filters which are not necessarily generic for the iteration. We finish the paper with an example of a symmetric iteration in which we subsume the work of Monro in \cite{Monro:1973} into this new framework, we discuss Kinna--Wagner Principles and construct a sequence of models, each an extension of the previous, where slowly more and more Kinna--Wagner Principles fail.

\subsection*{Motivation for the initiated} Before we begin, we would like to give some bits of motivation while ignoring any technical ignorance the reader might have at this point. If we try to iterate two symmetric extensions ``by hand'', say $\PP\ast\dot\QQ$, then we ultimately want to find a good definition for the class of names which are not only symmetric with respect to $\PP$, but then again symmetric with respect to $\QQ$ itself, this can be done awkwardly by hand for finite steps. The limit should be something akin to ``definable from finitely many objects so far'', which makes the ``by-hand'' approach much more difficult (since an object, even if guaranteed to exist, might have definitions and dependency on parameters according to the generic). Moreover, if one pays close attention to the details, there is a few immediate obstruction to the obvious way of applying the automorphisms even for the two steps.

So the goal is to try and generalize the structure of a symmetric extension: automorphisms, and then filters of groups. The main problem is that the structure from the first symmetric extension acts on the second one as well. We need to address these actions, and create a structure that is useful, even if somewhat complicated.

\section{Preliminaries}\label{section:prelim}
We work in $\ZFC$ throughout the entire paper. The main concern of this work is forcing, in this section we establish some notation and terminology. We say that $\tup{\PP,1_\PP,\leq_\PP}$ is a notion of forcing if $\PP$ is a preordered set and $1_\PP$ is a maximum element. The elements of $\PP$ are called conditions, and if $p,q$ are two conditions in $\PP$ we say that $q$ is \textit{stronger} than $p$ if $q\leq_\PP p$. We will usually omit $1_\PP$ from the definition, and when the context allows we will omit $\PP$ from the subscript as well.

We denote $\PP$-names as $\dot x$, and we adopt the convention that the elements of $\dot x$ are pairs $\tup{p,\dot y}$ where $p$ is a condition in $\PP$ and $\dot y$ is a $\PP$-name of a lower rank, indeed this lets us have a natural notion of rank on $\PP$-names, denoted by $\rank_\PP(\dot x)$ (or $\rank(\dot x)$ when $\PP$ is clear from context\footnote{While rank functions are fairly ubiquitous in set theory, and often refer to the von Neumann rank, we will almost always refer to the name ranks instead, and make any reference to other ranks explicit.}). If $x$ is in the ground model, we will denote by $\check x$ the canonical name for $x$, defined by recursion as $\check x=\{\tup{1,\check y}\mid y\in x\}$. We say that a condition $p$  \textit{appears} in $\dot x$, if there is some $\dot y$ such that $\tup{p,\dot y}\in\dot x$; similarly we say that $\dot y$ appears in $\dot x$, if there is some condition $p$ such that $\tup{p,\dot y}\in\dot x$.

\begin{definition}We say that a $\PP$-name $\dot x$ is \textit{open} if whenever $\tup{p,\dot y}\in\dot x$ and $q\leq p$, then $\tup{q,\dot y}\in\dot x$. We say that $\dot x$ is a \textit{simple} name if every $\dot y$ which appears in $\dot x$ is of the form $\check a$ for some $a$.
\end{definition}

\begin{proposition}Suppose that $\dot x$ is a name such that for some $A$, $p\forces\dot x\subseteq\check A$, then there is an open and simple name $\dot x_0$ such that $p\forces\dot x=\dot x_0$.
\end{proposition}
\begin{proof}
Define $\dot x_0=\{\tup{p,\check a}\mid p\forces\check a\in\dot x\}$, then by the properties of the forcing relation $\dot x_0$ is an open name, and by the definition of $\dot x_0$ it is also simple. If $p$ does not force equality, let $q\leq p$ such that $q\forces\dot x\neq\dot x_0$. Assume first there is some $r\leq q$ such that for some $\dot y$, $r\forces\dot y\in\dot x\land\dot y\notin\dot x_0$. By the assumption that $r\leq q\leq p$ we know that $r\forces\dot x\subseteq\check A$, then there is some $r'\leq r$ and $a\in A$ such that $r'\forces\dot y=\check a$. But now $r'\forces\check a\in\dot x$, and by definition $\tup{r',\check a}\in\dot x_0$ so in particular $r'\forces\check a\in\dot x_0$, which is a contradiction. The other inclusion is similar, and we get that no $q\leq p$ can force that $\dot x\neq\dot x_0$, so $p\forces\dot x=\dot x_0$.
\end{proof}

Suppose that $\{\dot x_i\mid i\in I\}$ is a collection of $\PP$-names, we write $\{\dot x_i\mid i\in I\}^\bullet$ for the $\PP$-name $\{\tup{1,\dot x_i}\mid i\in I\}$, and we call these \textit{$\bullet$-names}. Using this notation, for example, $\check x=\{\check y\mid y\in x\}^\bullet$. This notation extends naturally to ordered pairs and sequences, so $\tup{\dot x,\dot y}^\bullet=\{\{\dot x\}^\bullet,\{\dot x,\dot y\}^\bullet\}^\bullet$, and so on.

Let $\dot x$ be a $\PP$-name and $p\in\PP$, we will write $\dot x\restriction p$ to be the name $p$ interprets correctly as $\dot x$, but any incompatible condition interprets as the empty set. Namely,\[\dot x\restriction p = \{\tup{q,\dot y}\mid q\leq p,\ q\forces\dot y\in\dot x,\ \text{ and }\ \dot y\text{ appears in }\dot x\}.\]
One might wonder why this is not an inductive definition in which we take $\dot y\restriction p$ or even $\dot y\restriction q$. The matter of fact is that usually this is unneeded in most uses of the notation.

Suppose that $\dot x$ and $\dot A$ are $\PP$-names. We say that $\dot x$ was \textit{obtained by mixing} from names which appear in $\dot A$ if there is a maximal antichain $D$ and for every $p\in D$, some $\dot a_p$ which appears in $\dot A$ such that $\dot x\restriction p=\dot a_p\restriction p$.
\subsection{Iterated forcing}
In this work we will only consider finite support iterations. In different sources one can find slightly different presentations of finite support iterations. When we work in $\ZFC$ and we just want to take generic extensions, all these presentations of the iteration poset are equivalent. However, once you want to apply automorphisms to a forcing, the choice of the preordered set can make a difference. For example, it is easy to construct a finitely splitting tree which is rigid, but forcing equivalent to the full binary tree (in other words, a Cohen forcing).

\begin{definition}Suppose that $\PP$ is a notion of forcing and $\dot\QQ$ is a $\PP$-name such that $1_\PP\forces_\PP``\dot\QQ$ is a notion of forcing''. The \textit{Kunen iteration} is the preordered set given by $\{\tup{p,\dot q}\mid p\in\PP,\dot q\text{ appears in }\dot\QQ, p\forces_\PP\dot q\in\dot\QQ\}$; the \textit{Jech iteration} is the set \[\{\tup{p,\dot q}\mid p\in\PP,\dot q\text{ obtained by mixing names which appear in }\dot\QQ, 1_\PP\forces_\PP\dot q\in\dot\QQ\}.\]

In both cases, $\tup{p,\dot q}\leq\tup{p',\dot q'}$ if $p\leq_\PP p'$ and $p\forces\dot q\leq_{\dot\QQ}\dot q'$. We denote this forcing by $\PP\ast\dot\QQ$.
\end{definition}
It is a standard exercise to show that for a two-step iteration using Kunen or Jech iterations is the same. The proof allows us to extend this to any length of iteration, provided that the limit steps are taken with finite support. While we only consider Jech iterations in this work, the Kunen iterations will make an important appearance in \autoref{subsection:Kunen}.

One important remark that should be made here is that in most books concerning iterated forcing, it is often the case that the initial assumption is that without loss of generality $\dot\QQ$ is going to be a set of ordinals, or something like this in order to ensure some canonicity in the names for the conditions. If we plan on forcing over models where the axiom of choice fails, this is going to be a problem, as in all likelihood we will be interested in forcings which cannot be well-ordered. This is why in the Kunen iteration we require explicitly that $\dot q$ appears in $\dot\QQ$, and while in the Jech iterations we allowed a greater freedom, but we still ensure that the names were mixed from our original $\dot\QQ$.

We finish by pointing out that neither definition of the iteration preordered set is necessarily separative. It will often be convenient to jump back and forth between the presentation of $\PP\ast\dot\QQ$ as a set of ordered pairs, and the separative quotient. So we will allow ourselves to replace conditions which are equivalent when it is convenient, but only when it is clear that we are allowed to do so. We will also take advantage of the fact that if $p\forces_\PP\dot q=\dot q'$, then $\tup{p,\dot q}$ and $\tup{p,\dot q'}$ are equivalent in the sense of the separative quotient.

\subsection{Symmetric extensions}
Let $\PP$ be a notion of forcing, if $\pi\in\aut(\PP)$, we can extend it by recursion to an automorphism of names:
\[\pi\dot x=\{\tup{\pi p,\pi\dot y}\mid\tup{p,\dot y}\in\dot x\}.\]
\begin{lemma}[The Symmetry Lemma]\label{lemma:symmetry}
Suppose that $\PP$ is a notion of forcing, $\pi$ an automorphism of $\PP$ and $\dot x_1,\ldots,\dot x_n$ are $\PP$-names. Then for every $\varphi(u_1,\ldots,u_n)$ in the language of forcing,
\[p\forces\varphi(\dot x_1,\ldots,\dot x_n)\iff\pi p\forces\varphi(\pi\dot x_1,\ldots,\pi\dot x_n).\qed\]
\end{lemma}
The following proposition is an easy corollary of the Symmetry Lemma.
\begin{proposition}\label{prop:open names and symmetries}
Suppose that $\dot x$ is a $\PP$-name, and $\dot x'$ is the open name defined as \[\dot x'=\{\tup{p,\dot y}\mid p\forces\dot y\in\dot x\text{ and }\dot y\text{ appears in }\dot x\}.\] Then for every automorphism $\pi$, $1\forces\pi\dot x=\pi\dot x'$.\qed
\end{proposition}

If $\sG$ is a group, we say that $\sF$ is a \textit{filter of subgroups} over $\sG$ if $\sF$ is a non-empty collection of subgroups of $\sG$, closed under supergroups and finite intersections. If the trivial subgroup is in $\sF$ we say that it is the improper filter.\footnote{It is often the case that the improper filter is not considered a filter, but allowing this can be fruitful, as we will see in \autoref{section:toy}.} Such $\sF$ is called \textit{normal} if whenever $H\in\sF$ and $\pi\in\sG$, then the conjugated subgroup $\pi H\pi^{-1}\in\sF$ as well.

\begin{remark}
Note that it is enough to talk about a filter base, or rather a normal filter base, rather than talking about filters. This will be convenient later when we move between models and a filter in a small model might ``only'' be a filter base in the larger model. But we will largely ignore this problem throughout the text.
\end{remark}

\begin{definition}We say that $\tup{\PP,\sG,\sF}$ is a \textit{symmetric system} if $\PP$ is a notion of forcing, $\sG$ is a group of automorphisms of $\PP$ and $\sF$ is a normal filter of subgroups over $\sG$. We say that $\dot x$ is \textit{$\sF$-symmetric} if there exists $H\in\sF$ such that for all $\pi\in H$, $\pi\dot x=\dot x$. This definition extends by recursion: $\dot x$ is \textit{hereditarily $\sF$-symmetric}, if $\dot x$ is $\sF$-symmetric and every name which appears in $\dot x$ is hereditarily $\sF$-symmetric. We denote by $\HS_\sF$ the class of all hereditarily $\sF$-symmetric names.
\end{definition}
\begin{proposition}
Let $\dot x$ be an $\sF$-symmetric with $H\in\sF$ witnessing that, and let $\dot x'$ be as defined in \autoref{prop:open names and symmetries}. Then $H$ witnesses that $\dot x'$ is $\sF$-symmetric.\qed
\end{proposition}
\begin{theorem}
Suppose that $\tup{\PP,\sG,\sF}$ is a symmetric system and $G\subseteq\PP$ is a $V$-generic filter. Denote by $N$ the class $\HS_\sF^G=\{\dot x^G\mid\dot x\in\HS_\sF\}$, then $N\models\ZF$ and $V\subseteq N\subseteq V[G]$.\qed
\end{theorem}
To read more about symmetric extensions see \cite{Jech:AC}, and chapter 15 in \cite{Jech:ST}.
\section{Extending automorphisms}\label{section:ext aut}
\subsection{Two step iterations}
Through this entire section, $\PP$ will denote some notion of forcing, and $\dot\QQ$ will denote a name for a notion of forcing. Let us investigate how to extend automorphisms of $\PP$ and $\dot\QQ$ to automorphisms of $\PP\ast\dot\QQ$.
\begin{proposition}\label{prop:aut-ext}
Let $\dot\sigma$ a $\PP$-name such that $1_\PP\forces_\PP\dot\sigma\in\aut(\dot\QQ)$. Then the map $\tup{p,\dot q}\mapsto\tup{p,\dot\sigma(\dot q)}$ is an automorphism of $\PP\ast\dot\QQ$.
\end{proposition}
\begin{proof}
First we need to argue why this is a well-defined map, note that since $\dot\sigma$ is guaranteed to be an automorphism of $\dot\QQ$ and $\dot q$ is guaranteed to be a condition of $\dot\QQ$, we can find a pre-dense set where the value of $\dot\sigma(\dot q)$ is decided and mix over it. Moreover, since $\dot\sigma$ is a name for a well-defined function, if both $\dot q_0$ and $\dot q_1$ are forced by $1_\PP$ to be equal to $\dot\sigma(\dot q)$, then it is necessarily the case that $1_\PP\forces_\PP\dot q_0=\dot q_1$.

Injectivity follows from the same argument. If $\tup{p,\dot\sigma(\dot q)}=\tup{p,\dot\sigma(\dot q')}$, then by the fact $1_\PP\forces``\dot\sigma\text{ is injective}"$ it follows that $p\forces_\PP\dot q=\dot q'$, and therefore the two conditions are equivalent. The argument for surjectivity is also the same. So it remains to prove that the order is preserved.

Suppose that $\tup{p_1,\dot q_1}\leq_{\PP\ast\dot\QQ}\tup{p_2,\dot q_2}$, then $p_1\leq_\PP p_2$ and $p_1\forces_\PP\dot q_1\leq_{\dot\QQ}\dot q_2$. But because $p_1\forces_\PP\dot\sigma\in\aut(\dot\QQ)$ we have that \[p_1\forces_\PP\dot q_1\leq_{\dot\QQ}\dot q_2\leftrightarrow\dot\sigma(\dot q_1)\leq_{\dot\QQ}\dot\sigma(\dot q_2).\] Therefore $\tup{p_1,\dot q_1}\leq_{\PP\ast\dot\QQ}\tup{p_2,\dot q_2}$ if and only if $\tup{p_1,\dot\sigma(\dot q_1)}\leq_{\PP\ast\dot\QQ}\tup{p_2,\dot\sigma(\dot q_2)}$.
\end{proof}

One might feel there is a simple analogy to the case when taking $\pi\in\aut(\PP)$ and considering $\tup{p,\dot q}\mapsto\tup{\pi p,\dot q}$. But this is not going to be order preserving, because there is no guarantee that $\pi p\forces_\PP\dot q\leq_{\dot\QQ}\dot q'$ just by assuming $p\forces_\PP\dot q\leq_{\dot\QQ}\dot q'$. In fact there is no guarantee that $\pi p\forces_\PP\pi\dot q\leq_{\dot\QQ}\pi\dot q'$. The reason is simple, it might be the case that $\pi$ does not preserve $\dot\QQ$ or its order. So not every $\pi$ can do the job, just those which preserve enough of the information about $\dot\QQ$.

\begin{definition}Let $\dot A$ is a $\PP$-name and $\pi\in\aut(\PP)$. We say that $\pi$ \textit{respects} $\dot A$ if $1_\PP\forces\dot A=\pi\dot A$. When $\dot A$ is implicitly endowed with some structure (e.g.\ a notion of forcing), then we require the entire structure to be respected.
\end{definition}
Note that the collection of all automorphisms which respect $\dot A$ form a subgroup of $\aut(\PP)$. Also note that if $\pi$ respects $\dot A$, it will invariably respect anything reasonably definable from $\dot A$ (like the name for a power set of $\dot A$, and so on).

\begin{proposition}If $\pi\in\aut(\PP)$ respects $\dot\QQ$, then the map $\tup{p,\dot q}\mapsto\tup{\pi p,\pi\dot q}$ is an automorphism of $\PP\ast\dot\QQ$.\qed
\end{proposition}

Now, if $\tup{\pi,\dot\sigma}$ is a pair such that $\pi\in\aut(\PP)$ respects $\dot\QQ$ and $1_\PP\forces_\PP\dot\sigma\in\aut(\dot\QQ)$, we define the action of $\tup{\pi,\dot\sigma}$ on $\PP\ast\dot\QQ$ as the successive operation of $\dot\sigma$ and then $\pi$. We introduce a notation\footnote{We will soon start applying these automorphisms to the names of other automorphisms. It is useful to have a distinction between the names defining the automorphism and the action they induce. The integral symbol was chosen because (1) it is not used in set theory, and (2) it is reasonably looking with subscripts.} to simplify this,\[\gaut{\tup{\pi,\dot\sigma}}\tup{p,\dot q}=\tup{\pi p,\pi(\dot\sigma(\dot q))}=\tup{\pi p,\pi\dot\sigma(\pi\dot q)}.\]
We will also denote, for simplicity, $\gaut{\pi}$ and $\gaut{\dot\sigma}$ the extensions of automorphisms from $\aut(\PP)$ and $\aut(\dot\QQ)$ respectively. So $\gaut{\tup{\pi,\dot\sigma}}=\gaut{\pi}\gaut{\dot\sigma}$.

\begin{proposition}The following properties are true for $\tup{\pi_0,\dot\sigma_0}$ and $\tup{\pi_1,\dot\sigma_1}$:
\begin{enumerate}
\item $\gaut{\tup{\pi_1,\dot\sigma_1}}\gaut{\tup{\pi_0,\dot\sigma_0}}=\gaut{\tup{\pi_1\pi_0,\pi_0^{-1}(\dot\sigma_1)\dot\sigma_0}}$.
\item $\gaut{\tup{\pi_1,\dot\sigma_1}}^{-1}=\gaut{\tup{\pi_1^{-1},\pi_1(\dot\sigma_1^{-1})}}$.
\end{enumerate}
\end{proposition}
\begin{proof}
First we will prove $(1)$:
\begin{align*}
\gaut{\tup{\pi_1,\dot\sigma_1}}\gaut{\tup{\pi_0,\dot\sigma_0}}\tup{p,\dot q} 
&=\gaut{\tup{\pi_1,\dot\sigma_1}}\tup{\pi_0 p,\pi_0(\dot\sigma_0(\dot q))}\\
&=\tup{\pi_1\pi_0 p,\pi_1(\dot\sigma_1(\pi_0(\dot\sigma_0(\dot q))))}\\
&=\tup{\pi_1\pi_0 p,\pi_1\pi_0((\pi_0^{-1}(\dot\sigma_1)\dot\sigma_0)(\dot q))}\\
&=\gaut{\tup{\pi_1\pi_0,\pi_0^{-1}(\dot\sigma_1)\dot\sigma_0}}\tup{p,\dot q}
\end{align*}
Now $(2)$ follows easily, as we only need to argue that $\gaut{\tup{\pi_1,\dot\sigma_1}}\gaut{\tup{\pi_1^{-1},\pi_1(\dot\sigma_1^{-1})}}=\id$, and indeed,
\[\gaut{\tup{\pi_1,\dot\sigma_1}}\gaut{\tup{\pi_1^{-1},\pi_1(\dot\sigma_1^{-1})}}=\gaut{\tup{\pi_1\pi_1^{-1},\pi_1(\dot\sigma_1)\pi_1(\dot\sigma_1^{-1})}}=\gaut{\tup{\id,\pi_1(\id^\bullet)}}=\id.\qedhere\]
\end{proof}

The readers who are knowledgeable in group theory might see something familiar here: this operation looks suspiciously close to the definition of a semi-direct product. And those readers will not be wrong. The major difference here is that due to the fact $\dot\sigma$ is a name for an automorphism, we are interested in $\gaut{\dot\sigma}$ as an automorphism of the iteration, but also in how $\dot\sigma$ will turn out in the generic extension. And it might be that $1_\PP\nforces\dot\sigma=\dot\tau$, but there is some $p$ such that $p\forces\dot\sigma=\dot\tau$, so in different generic extensions $\dot\sigma$ will become different automorphisms.

Before proceeding, one might wonder why did we choose $\gaut{\pi}\gaut{\dot\sigma}$ as the order of operation, rather than $\gaut{\dot\sigma}\gaut{\pi}$.
\begin{proposition}\label{prop:almost-commutative}
The following holds for $\tup{\pi,\dot\sigma}$: $\gaut{\dot\sigma}\gaut{\pi}=\gaut{\pi,\pi^{-1}\dot\sigma}$.
\end{proposition}
\begin{proof}
Note that $\gaut{\dot\sigma}$ is really $\gaut{\tup{\id,\dot\sigma}}$ and likewise $\gaut{\pi}=\gaut{\tup{\pi,\id^\bullet}}$ so by the previous proposition, $\gaut{\tup{\id,\dot\sigma}}\gaut{\tup{\pi,\id^\bullet}}=\gaut{\tup{\id\, \pi,\pi^{-1}\dot\sigma\, \id^\bullet}}=\gaut{\tup{\pi,\pi^{-1}\dot\sigma}}$.
\end{proof}
This means that the order itself is ``almost irrelevant'', at least for finite support iterations, and by paying a small ``price of conjugation'' we can switch the order to our liking. 

\begin{definition}[Generic semi-direct product]
Let $\dot\sG_1$ be a name for a group of automorphisms of $\dot\QQ$, and let $\sG_0$ be a group of automorphisms of $\PP$ which respect both $\dot\QQ$ and $\dot\sG_1$. We define the \textit{generic semi-direct product} $\sG_0\ast\dot\sG_1$ to be the following group of automorphisms of $\PP\ast\dot\QQ$:
\[\sG_0\ast\dot\sG_1=\left\{\gaut{\tup{\pi,\dot\sigma}}\middd\begin{array}{ll}
\pi\in\sG_0, \\1_\PP\forces\dot\sigma\in\dot\sG_1,\text { and }\\ \dot\sigma\text{ was obtained by mixing over names appearing in }\dot\sG_1
\end{array}\right\}.\]
\end{definition}
To see that $\sG_0\ast\dot\sG_1$ is indeed a group one can observe that because we required that $\sG_0$ respects $\dot\sG_1$, the composition and inverses can still be obtained by names that were mixed from $\dot\sG_1$.
\begin{proposition}The group $\sG_0\ast\dot\sG_1$ is generated by $\{\gaut{\pi},\gaut{\dot\sigma}\mid\pi\in\sG_0,1_\PP\forces_\PP\dot\sigma\in\dot\sG_1\}$ as a subgroup of $\aut(\PP\ast\dot\QQ)$.
\end{proposition}
\begin{proof}
One inclusion is trivial, as by definition $\gaut{\tup{\pi,\dot\sigma}}=\gaut{\pi}\gaut{\dot\sigma}$. In the other direction note that for every $\pi\in\sG_0$ the automorphism $\gaut{\tup{\pi,\id^\bullet}}$ satisfies that $\pi\in\sG_0$ and $1_\PP\forces_\PP\id^\bullet\in\dot\sG_1$, and similarly for $\gaut{\tup{\id,\dot\sigma}}$ for every $\dot\sigma$ which was obtained by mixing over names appearing in $\dot\sG_1$.
\end{proof}

\subsection{The general case}
\begin{definition}Let $\tup{\dot\QQ_\alpha\mid\alpha<\delta},\tup{\PP_\alpha\mid\alpha\leq\delta}$ be a finite support iteration, and for each $\alpha$, $\dot\sG_\alpha$ is a name for an automorphism group of $\dot\QQ_\alpha$ such that the following holds:
\begin{enumerate}
\item $\cG_0=\{\id\}$.
\item $\cG_{\alpha+1}=\cG_\alpha\ast\dot\sG_\alpha$.
\item If $\alpha$ is a limit ordinal, then $\cG_\alpha$ is the direct limit of the groups $\cG_\beta$.\footnote{Note that for $\beta\leq\beta'$ there is a natural embedding of $\cG_\beta$ into $\cG_{\beta'}$, thus the notion of a direct limit makes natural sense.}
\end{enumerate}
Moreover, for each $\alpha$, $\dot\QQ_\alpha$ and $\dot\sG_\alpha$ are respected by all the automorphisms in $\cG_\alpha$. Then $\cG_\delta$ is the \textit{generic semi-direct product} of the $\dot\sG_\alpha$ and it is an automorphism group of $\PP_\delta$.
\end{definition}
Much like the presentation of the two-step generic semi-direct product, we can represent the automorphisms in $\cG_\delta$ in a canonical way, as shown in the next proposition.
\begin{proposition}
Every automorphism in $\cG_\delta$ can be represented by a finite sequence of $\dot\pi_{\alpha_i}$, for $\alpha_0<\ldots<\alpha_n<\delta$ such that $1_{\alpha_i}\forces_{\alpha_i}\dot\pi_{\alpha_i}\in\dot\sG_{\alpha_i}$, and the automorphism is obtained by $\gaut{\dot\pi_{\alpha_0}}\cdots\gaut{\dot\pi_{\alpha_n}}$.
\end{proposition}
\begin{proof}
We prove this by induction on $\delta$. If $\delta\leq 1$ this is just by definition, and if $\delta$ is a limit ordinal it follows immediately by the fact that $\cG_\delta$ is a direct limit of $\cG_\alpha$ for $\alpha<\delta$. Suppose that $\delta=\alpha+1$, then every automorphism in $\cG_\delta$ is obtained as a pair $\tup{\pi,\dot\pi_\alpha}$ where $\pi$ is an automorphism in $\cG_\alpha$, which is what we wanted to show.
\end{proof}
We can now use this proposition to represent our automorphisms in a rather canonical way.
\begin{definition}
We say that a sequence $\vec\pi=\tup{\dot\pi_\alpha\mid\alpha<\delta}$ is an \textit{automorphism sequence} if
\begin{enumerate}
\item For every $\alpha<\delta$, $1_\alpha\forces_\alpha\dot\pi_\alpha\in\dot\sG_\alpha$.
\item $C(\vec\pi)\subseteq\delta$ is a finite set such that for every $\alpha\notin C(\vec\pi)$, $1_\alpha\forces_\alpha\dot\pi_\alpha=\id^\bullet$.
\end{enumerate}
We write $\gaut{\vec\pi}$ to denote the automorphism in $\cG_\delta$ which is induced by $\vec\pi$. We will also write $\vec\pi^{-1}$ for the sequence such that $\gaut{\vec\pi^{-1}}=\gaut{\vec\pi}^{-1}$, and $\vec\pi\circ\vec\sigma$ for the sequence for which $\gaut{\vec\pi\circ\vec\sigma}=\gaut{\vec\pi}\gaut{\vec\sigma}$ (note that these are not the pointwise inverse or composition, and we will calculate them explicitly in the next subsection).
\end{definition}
To get a better intuition as to how $\gaut{\vec\pi}$ acts on $\PP_\delta$, let us first consider what happens when $C(\vec\pi)=\{\alpha\}$. The action is defined to be as follows
\[\gaut{\dot\pi_\alpha}p = p\restriction\alpha^\smallfrown\dot\pi_\alpha(p(\alpha))^\smallfrown(\dot\pi_\alpha(p\restriction(\alpha,\delta))).\]
So when we apply $\gaut{\dot\pi_\alpha}$ we factor $\PP_\delta$ into three parts, $\PP_\alpha\ast(\dot\QQ_\alpha\ast\PP_\delta{/}\PP_{\alpha+1})$, then $\dot\pi_\alpha$ acts on $\dot\QQ_\alpha$ as an automorphism of the forcing, and on $\PP_\delta{/}\PP_{\alpha+1}$ as $\dot\QQ_\alpha$-names, and of course it does nothing to $\PP_\alpha$ itself. Now if we have $\gaut{\vec\pi}$, then we simply apply it one step at a time from the maximal element of $C(\vec\pi)$ to the minimal.

\begin{remark}
We use (again) the fact that $\PP_\delta$ is a finite support iteration to ensure that such decomposition exists. If one wishes to extend this definition to larger supports, then one needs to overcome this obstacle. For example a Revised Countable Support iteration might work if one wishes to extend the definition to countable support, however in such context the automorphisms should be applied ``from the bottom up'' rather than the ``top to bottom'' approach that we take here. It is not entirely clear, however, that such construction will work or generalize further (e.g.\ Easton support iterations).
\end{remark}
\subsection{Properties of the generic semi-direct product}
Let $\tup{\dot\QQ_\alpha,\dot\sG_\alpha\mid\alpha<\delta}$ be a finite support iteration as before. We will investigate some properties of $\cG_\delta$.
\begin{proposition}\label{prop:generic equality of aut}
Suppose that $\vec\pi$ and $\vec\sigma$ are two sequences of automorphisms, and $p\in\PP_\delta$ such that $p\forces_\delta\dot\pi_\alpha=\dot\sigma_\alpha$ for every $\alpha<\delta$. Then $\gaut{\vec\pi}p=\gaut{\vec\sigma}p$. In particular both automorphisms behave the same on the cone below $p$.
\end{proposition}
To shorten the statements, we will write $p\forces_\delta\vec\pi=\vec\sigma$ as a shorthand for ``$p\forces_\delta\dot\pi_\alpha=\dot\sigma_\alpha$ for all $\alpha<\delta$''.
\begin{proof}
We prove this by induction on $n=|C(\vec\pi)\cup C(\vec\sigma)|$. Of course this is trivial for $n=0$ as that means both $\vec\pi$ and $\vec\sigma$ are the identity.

Let $\alpha=\min C(\vec\pi)\cup C(\vec\sigma)$, then $p\restriction\alpha\forces_\alpha\dot\pi_\alpha=\dot\sigma_\alpha$. Consider $\vec\pi\restriction(\alpha,\delta)$ and $\vec\sigma\restriction(\alpha,\delta)$ as $\PP_\alpha$-names for $\dot\QQ_\alpha$-names, then we have that 
\[p\restriction\alpha\forces_\alpha``p(\alpha)\forces_{\QQ_\alpha}``p\restriction(\alpha,\delta)\forces_{\delta/\alpha+1}\vec\pi\restriction(\alpha,\delta)=\vec\sigma\restriction(\alpha,\delta)"".\]
Therefore we can apply $\dot\pi_\alpha$ and $\dot\sigma_\alpha$ on the internal forcing statement and receive the same results, since forcing below $p\restriction\alpha$ both are the same. This in turn translates to the statement that
\[\gaut{\dot\sigma_\alpha}p\restriction\alpha+1=\gaut{\dot\pi_\alpha}p\restriction\alpha+1\forces_{\alpha+1}\gaut{\dot\pi_\alpha}\vec\pi\restriction(\alpha,\delta)=\gaut{\dot\sigma_\alpha}\vec\sigma\restriction(\alpha,\delta).\]
Let $p_*=\gaut{\dot\pi_\alpha}p\restriction\alpha+1$, let $\vec\pi_*$ be $\gaut{\dot\pi_\alpha}\vec\pi\restriction(\alpha,\delta)$ and similarly let $\vec\sigma_*$ be $\gaut{\dot\sigma_\alpha}\vec\sigma\restriction(\alpha,\delta)$.  We have that $|C(\vec\pi_*)\cup C(\vec\sigma_*)|=n-1<n$, and that $p_*\forces_\delta\vec\pi_*=\vec\sigma_*$, so by the induction hypothesis we get that $\gaut{\vec\pi_*}p_*=\gaut{\vec\sigma_*}p_*$, but since $\gaut{\vec\pi}p=\gaut{\vec\pi_*}p_*$ and $\gaut{\vec\sigma}p=\gaut{\vec\sigma_*}p_*$ we get the wanted equality.
\end{proof}
The above proposition is the reason the automorphism group is called generic. As we progress in our forcing towards a generic filter new equality between sequences can be added making them essentially the same from that point onward.
\begin{proposition}\label{prop:preserving names}
If $p\forces_\delta\vec\pi=\vec\sigma$ and $\gaut{\vec\pi}p=p$, then for every $\PP_\delta$-name, $\dot x$, $p\forces_\delta\gaut{\vec\pi}\dot x=\gaut{\vec\sigma}\dot x$.
\end{proposition}
\begin{proof}
By the previous proposition it follows that $\gaut{\vec\sigma}p=p$, and that for every $q\leq_\delta p$ it is also true that $\gaut{\vec\pi}q=\gaut{\vec\sigma}q\leq_\delta\gaut{\vec\sigma}p=p$. In other words, the cone below $p$ is closed under $\gaut{\vec\pi}$ and $\gaut{\vec\sigma}$, and the two are equal there. It follows that their inverses are also equal on the cone below $p$ (whether or not if the sequences $\vec\pi^{-1}$ and $\vec\sigma^{-1}$ are forced to be equal below $p$). 

We prove the claim by induction on the rank of $\dot x$. Let $q\leq_\delta p$ such that $q\forces_\delta\dot z\in\gaut{\vec\pi}\dot x$. We may assume without loss of generality that for some $\tup{s,\dot u}\in\gaut{\vec\pi}\dot x$ we have that $q\leq_\delta s$ and $q\forces_\delta\dot z=\dot u$. Thus, $\gaut{\vec\pi^{-1}}q\forces_\delta\gaut{\vec\pi^{-1}}\dot u\in\dot x$. By the induction hypothesis, and as the rank of $\dot u$ is lower, $p\forces_\delta\gaut{\vec\sigma}\gaut{\vec\pi^{-1}}\dot u=\gaut{\vec\pi}\gaut{\vec\pi^{-1}}\dot u=\dot u$. From the assumption that $p\forces_\delta\vec\pi=\vec\sigma$, we also get that $\gaut{\vec\sigma}\gaut{\vec\pi^{-1}}q=\gaut{\vec\pi}\gaut{\vec\pi^{-1}}q=q$. 

Combining both these things we get that $q\forces_\delta\dot u\in\gaut{\vec\sigma}\dot x$, and so we obtain the inclusion $p\forces_\delta\gaut{\vec\pi}\dot x\subseteq\gaut{\vec\sigma}\dot x$. The other inclusion is proved similarly, and so $p\forces_\delta\gaut{\vec\pi}\dot x=\gaut{\vec\sigma}\dot x$.
\end{proof}
Examining the above proof in detail will expose the necessity of $\gaut{\vec\pi}p=p$, and why this is not a trivial statement. If we only know that $p\forces_\delta\vec\pi=\vec\sigma$, there is no guarantee that we can apply $\gaut{\vec\pi^{-1}}$ or $\gaut{\vec\sigma}\gaut{\vec\pi^{-1}}$ so judiciously and retain equality. Namely, there is no guarantee that $\gaut{\vec\pi}p\forces_\delta\vec\pi^{-1}=\vec\sigma^{-1}$.\footnote{Is the failure of this even possible? We do not have the answer. It is somewhat irrelevant too, after introducing the concept of tenacity. However, this might be somewhat natural to ponder, and to some extent provides motivation for the discussion about tenacity to begin with.}
\begin{proposition}\label{prop:aut seq calculations}
If $\vec\pi$ and $\vec\sigma$ are two sequences of automorphisms, then the following is true:
\begin{enumerate}
\item For $\alpha<\delta$, $\vec\pi^{-1}\restriction\alpha=(\vec\pi\restriction\alpha)^{-1}$.
\item $\vec\pi\circ\vec\sigma=\tup{(\gaut{\vec\sigma^{-1}}\dot\pi_\alpha)\dot\sigma_\alpha\mid\alpha<\delta}$.
\item $\vec\pi^{-1}=\tup{\gaut{\vec\pi}\dot\pi_\alpha^{-1}\mid\alpha<\delta}$.
\end{enumerate}
\end{proposition}
\begin{proof}We prove all three simultaneously by induction on $\delta$. 

For $(1)$, we prove this by induction on $\alpha$ and by noting the following equality holds
\[\gaut{\vec\pi\restriction(\alpha+1)^{-1}}=\gaut{\dot\pi_\alpha^{-1}}\gaut{\vec\pi\restriction\alpha^{-1}}=\gaut{\vec\pi\restriction\alpha^{-1}}\gaut{\gaut{\vec\pi\restriction\alpha}(\dot\pi_\alpha^{-1})}.\]
(The last equality is exactly because we consider $\cG_{\alpha+1}$ as a two-step generic semi-direct product and applying \autoref{prop:almost-commutative}.)

For $(2)$, First note that if $C(\vec\pi)=\{\alpha\}$ and $\alpha>\max C(\vec\sigma)$, then $\gaut{\vec\pi}\gaut{\vec\sigma}=\gaut{\vec\sigma}\gaut{\gaut{\vec\sigma^{-1}}\vec\pi}$. Now we can prove $(2)$ by induction on $\alpha=\max(C(\vec\pi)\cup C(\vec\sigma))$:
\[\gaut{\vec\pi}\gaut{\vec\sigma}=\gaut{\vec\pi\restriction\alpha^\smallfrown\dot\pi_\alpha}\gaut{\vec\sigma\restriction\alpha^\smallfrown\dot\sigma_\alpha}=
\gaut{\vec\pi\restriction\alpha}\gaut{\dot\pi_\alpha}\gaut{\vec\sigma\restriction\alpha}\gaut{\dot\sigma_\alpha}=
\gaut{\vec\pi\restriction\alpha}\gaut{\vec\sigma\restriction\alpha}\gaut{\gaut{\vec\sigma\restriction\alpha^{-1}}\dot\pi_\alpha}\gaut{\dot\sigma_\alpha}.\]
The last equality follows from \autoref{prop:almost-commutative} as in the case of $(1)$. Using the induction hypothesis, we get that $\vec\pi\restriction\alpha\circ\vec\sigma\restriction\alpha$ equals to $\tup{(\gaut{\vec\sigma\restriction\alpha^{-1}}\dot\pi_\beta)\dot\sigma_\beta\mid\beta<\alpha}$ and $\gaut{\gaut{\sigma\restriction\alpha^{-1}}\dot\pi_\alpha}\gaut{\dot\sigma_\alpha}=\gaut{\left(\gaut{\vec\sigma\restriction\alpha^{-1}}\dot\pi_\alpha\right)\dot\sigma_\alpha}$ is the last nontrivial coordinate of $\vec\pi\circ\vec\sigma$. Applying $(1)$ finishes the proof, since $\vec\sigma\restriction\alpha^{-1}=\vec\sigma^{-1}\restriction\alpha$.

Finally, to prove $(3)$ we use $(2)$, denote by $\vec\sigma=\tup{\gaut{\vec\pi}\dot\pi_\alpha^{-1}\mid\alpha<\delta}$. We will show that $\vec\sigma\circ\vec\pi=\id$:
\[\vec\sigma\circ\vec\pi=\tup{(\gaut{\vec\pi^{-1}}(\gaut{\vec\pi}\dot\pi_\alpha^{-1}))\dot\pi_\alpha\mid\alpha<\delta}=\tup{\dot\pi_\alpha^{-1}\dot\pi_\alpha\mid\alpha<\delta}=\tup{\id^\bullet\mid\alpha<\delta}.\qedhere\]
\end{proof}
We now get two corollaries from the above computations and \autoref{prop:preserving names}.
\begin{corollary}\label{cor:tenacious composition}
Suppose that $p\forces_\delta\vec\pi=\vec\tau$ and $\gaut{\vec\pi}p=p$, then the following is true:
\begin{enumerate}
\item $p\forces_\delta\vec\pi^{-1}=\vec\tau^{-1}$.
\item For every $\vec\sigma$, $p\forces_\delta\vec\sigma\circ\vec\pi=\vec\sigma\circ\vec\tau$.
\end{enumerate}
\end{corollary}
\begin{proof}
To prove (1) first observe that $p\forces_\delta\dot\pi_\alpha=\dot\tau_\alpha$, and therefore $p\forces_\delta\dot\pi_\alpha^{-1}=\dot\tau_\alpha^{-1}$. Using the Symmetry Lemma, and the assumption that $\gaut{\vec\pi}p=\gaut{\vec\tau}p=p$, we get that $p\forces_\delta\gaut{\vec\pi}\dot\pi_\alpha^{-1}=\gaut{\vec\pi}\dot\tau_\alpha^{-1}\text{ and }\gaut{\vec\tau}\dot\pi_\alpha^{-1}=\gaut{\vec\tau}\dot\tau_\alpha^{-1}$.

Now applying \autoref{prop:preserving names} we also get that $p\forces_\delta\gaut{\vec\pi}\dot\pi_\alpha^{-1}=\gaut{\vec\tau}\dot\pi_\alpha^{-1}$. Therefore we obtain that for every $\alpha<\delta$,
\[p\forces_\delta\gaut{\vec\pi}\dot\pi_\alpha^{-1}=\gaut{\vec\tau}\dot\pi_\alpha^{-1}=\gaut{\vec\tau}\dot\tau_\alpha^{-1},\]
which is exactly the same as saying that $p\forces_\delta\vec\pi^{-1}=\vec\tau^{-1}$.

From this (2) follows immediately, as we have that $\vec\sigma\circ\vec\pi=\tup{(\gaut{\vec\pi^{-1}}\dot\sigma_\alpha)\dot\pi_\alpha\mid\alpha<\delta}$ and $\vec\sigma\circ\vec\tau=\tup{(\gaut{\vec\tau^{-1}}\dot\sigma_\alpha)\dot\tau_\alpha\mid\alpha<\delta}$. So by the assumption, (1) and \autoref{prop:preserving names}, we get that $p\forces_\delta(\gaut{\vec\pi^{-1}}\dot\sigma_\alpha)\dot\pi_\alpha=(\gaut{\vec\tau^{-1}}\dot\sigma_\alpha)\dot\tau_\alpha$ for all $\alpha<\delta$, as wanted.
\end{proof}
We finish this section by showing that the requirement on each step of the iteration being respected by the previous aggregate of automorphisms is a very nontrivial requirement. 
\begin{definition}We say that a notion of forcing $\PP$ is \textit{weakly homogeneous} if whenever $p,q\in\PP$ there is some $\pi\in\aut(\PP)$ such that $\pi p$ is compatible with $q$. If $\sG$ is a subgroup of $\aut(\PP)$ we say that \textit{$\sG$ witnesses the homogeneity of $\PP$} if we can find such $\pi$ in $\sG$.
\end{definition}
\begin{theorem}
Suppose that $\tup{\dot\QQ_\alpha,\dot\sG_\alpha\mid\alpha<\delta},\tup{\PP_\alpha,\cG_\alpha\mid\alpha\leq\delta}$ are as before, and that for every $\alpha<\delta$, $1_\alpha\forces_\alpha\dot\sG_\alpha\text{ witnesses the homogeneity of }\dot\QQ_\alpha$. Then $\cG_\delta$ witnesses the homogeneity of $\PP_\delta$.
\end{theorem}
This is a nontrivial theorem because the iteration of two weakly homogeneous forcings need not be weakly homogeneous itself. For example, $\PP$ chooses $0$ or $1$, and depending on the result $\dot\QQ$ is either trivial or chooses again $0$ or $1$, is an iteration where $\PP$ is weakly homogeneous and $\dot\QQ$ is forced to be weakly homogeneous, but the iteration itself is not.
\begin{proof}[Sketch of Proof]
Here we take advantage of the consequence of \autoref{prop:almost-commutative} that given any automorphism sequence $\vec\pi$ with $C(\vec\pi)=\{\alpha_1,\ldots,\alpha_n\}$ there are names $\dot\sigma_{\alpha_i}$ such that \[\gaut{\vec\pi}=\gaut{\dot\pi_{\alpha_1}}\cdots\gaut{\dot\pi_{\alpha_n}}=\gaut{\dot\sigma_{\alpha_n}}\cdots\gaut{\dot\sigma_{\alpha_1}}.\]

Suppose that $p,q\in\PP_\delta$ then there is a finite subset of $\delta$, some $\{\alpha_1,\ldots,\alpha_n\}$ such that both $p$ and $q$ have their support included in this finite set. Now we proceed by recursion on $i\leq n$ to make $p(\alpha_i)$ compatible with $q(\alpha_i)$, and this defines the wanted automorphism sequence.
\end{proof}
\section{Combining filters and the notion of supports}\label{section:ext filters}
In a symmetric system we have three objects: the forcing, the automorphism group and the filter of subgroups. The iteration of the forcings is just the finite support iteration, and we saw how to extend the automorphisms when we assume some additional (and necessary condition). Now we would like to extend the filters. However we have a problem if we want to preserve the normality of the filters, and we do. 

Consider the conjugation $\vec\pi\circ\vec\sigma\circ\vec\pi^{-1}$. By \autoref{prop:aut seq calculations} we can calculate this sequence to be the following sequence:
\begin{align*}
\vec\pi\circ\vec\sigma\circ\vec\pi^{-1}
&=\tup{(\gaut{\vec\sigma^{-1}}\dot\pi_\alpha)\dot\sigma_\alpha\mid\alpha<\delta}\circ\tup{\gaut{\vec\pi}\dot\pi^{-1}_\alpha\mid\alpha<\delta}\\
&=\tup{\gaut{\vec\pi}((\gaut{\vec\sigma^{-1}}\dot\pi_\alpha)\dot\sigma_\alpha)\gaut{\vec\pi}\dot\pi_\alpha^{-1}\mid\alpha<\delta}\label{eqn:conjugation}\tag{\textbf{C1}}\\
&=\gaut{\vec\pi}\tup{(\gaut{\vec\sigma^{-1}}\dot\pi_\alpha)\dot\sigma_\alpha\dot\pi_\alpha^{-1}\mid\alpha<\delta}
\end{align*}
Unfortunately this does not look like pointwise conjugation, unless $\gaut{\vec\sigma^{-1}}\dot\pi_\alpha=\dot\pi_\alpha$, and even then we still have that $\gaut{\vec\pi}$ acting outside the sequence. And this means that if we want to ensure that in some sense large groups can be conjugated and remain inside the ``iteration of the filters'', then some nontrivial assumptions are needed. While there is no escape of these additional assumptions, we chose to place them elsewhere in order to make the notion of a ``large subgroup'' slightly clearer.

\begin{definition}
Suppose that $\tup{\dot\QQ_\alpha,\dot\sG_\alpha\mid\alpha<\delta}$ is a finite support iteration system with the same requirements as in the previous section, and in addition that for every $\alpha<\delta$ we have a $\PP_\alpha$-name, $\dot\sF_\alpha$ such that the following holds:
\begin{enumerate}
\item $1_\alpha\forces_\alpha\dot\sF_\alpha\text{ is a normal filter of subgroups of }\dot\sG_\alpha$.
\item $\dot\sF_\alpha$ is respected by all the automorphisms in $\cG_\alpha$.
\end{enumerate}
We say that $\vec H=\tup{\dot H_\alpha\mid\alpha<\delta}$ is an \textit{$\cF_\delta$-support} if the following conditions hold:
\begin{enumerate}
\item For every $\alpha<\delta$, $\dot H_\alpha$ is a $\PP_\delta$-name.
\item $1_\delta\forces_\delta\dot H_\alpha\in\dot\sF_\alpha$.
\item $1_\delta\forces_\delta\{\alpha<\delta\mid\dot H_\alpha\neq\dot\sG_\alpha\}\text{ is finite}$.
\end{enumerate}
\end{definition}
The idea behind the definition is two-fold. First of all, we want the supports to be closed under mixing, it will prove to be both useful and essential as we progress. The second is that we want to allow some genericity regarding what would be the groups which support a name from each step. In the case of a limit step, we would have liked to say that something will turn out in the intermediate model if it was ``more or less'' definable from finitely many objects until that limit steps. But even if we know the definition, we have no way of knowing in the ground model, in advance, what are these objects and from which steps of the iterations they brought into existence. Allowing some genericity in the choice of the finite set of parameters, as well the parameters themselves, grants us some freedom when we come to apply mixing in our arguments.

If the context is clear, we will omit $\cF_\delta$ from the notation and terminology. We will also write $p\forces_\delta\vec\pi\in\vec H$, if $p\forces_\delta\dot\pi_\alpha\in\dot H_\alpha$, for every $\alpha<\delta$, and similarly $p\forces_\delta\vec H\subseteq\vec H'$. In the same spirit we will write $\vec H\cap\vec H'$ to denote the names for pointwise intersection of the two supports, and similarly for other set theoretic operations.

\begin{proposition}\label{prop:correctness}
Suppose that $\vec H$ is a support and $q\leq_\delta p$. Then whenever $\vec\pi$ is such that $q\forces_\delta\vec\pi\in\vec H$, there is some $\vec\sigma$ such that $p\forces_\delta\vec\sigma\in\vec H$ and $q\forces_\delta\vec\pi=\vec\sigma$.
\end{proposition}
\begin{proof}
We define $\vec\sigma$ coordinate-wise. For every $\alpha$, let $\dot\sigma_\alpha$ be a name, constructed by mixing for example, such that $q\restriction\alpha\forces_\alpha\dot\sigma_\alpha=\dot\pi_\alpha$, and any condition $r\in\PP_\alpha$ incompatible with $q\restriction\alpha$ satisfies $r\forces_\alpha\dot\sigma_\alpha=\id^\bullet$.

If $\alpha\notin C(\vec\pi)$, then we necessarily defined $\dot\sigma_\alpha$ to be such that $1_\alpha\forces_\alpha\dot\sigma_\alpha=\id^\bullet$, therefore $C(\vec\sigma)\subseteq C(\vec\pi)$. If $\alpha\in C(\vec\pi)$ and therefore $1_\alpha\forces_\alpha\dot\sigma_\alpha\in\{\id^\bullet,\dot\pi_\alpha\}^\bullet$. In either case $1_\delta\forces_\delta\vec\sigma\in\vec H$, so $p\forces_\delta\vec\sigma\in\vec H$.
\end{proof}
\begin{definition}

We say that a support $\vec H=\tup{\dot H_\alpha\mid\alpha<\delta}$ is an \textit{excellent support} if:
\begin{enumerate}
\item For every $\alpha$, $\dot H_\alpha$ is a $\PP_\alpha$-name.
\item The set $C(\vec H)=\{\alpha<\delta\mid 1_\alpha\nforces_\alpha\dot H_\alpha=\dot\sG_\alpha\}$ is finite.
\end{enumerate}
\end{definition}
Namely, if a support is some object which is somewhat generic in the choice of nontrivial coordinates, then an excellent support is much closer to being a concrete object in the ground model. The following is an obvious observation.
\begin{proposition}
If $\vec H$ is an $\cF_\delta$-support, then there is a dense open set $D\subseteq\PP_\delta$, such that for every $p\in D$ there is an excellent $\cF_\delta$-support, $\vec H_p$ and $p\forces_\delta\dot H_\alpha=\dot H_{p,\alpha}$ for every $\alpha<\delta$.\qed
\end{proposition}

If $\vec H$ is a $\cF_\delta$-support and $p\in\PP_\delta$, we will use $\vec H(p)$ to denote the group generated by $\{\gaut{\vec\pi}\mid p\forces_\delta\vec\pi\in\vec H\}$. Note that if $q\leq_\delta p$, then $\vec H(p)$ is a subgroup of $\vec H(q)$, and they act the same on the cone below $q$, especially if every $\gaut{\vec\pi}\in\vec H(q)$ satisfies $\gaut{\vec\pi}q=q$.

We can therefore view $\vec H(p)$ as some sort of an approximation to the group $\vec H$ in the full generic extension. And in fact, we will see later that these form something that looks a bit like a normal filter of subgroups, in a generic sense. This will give us a natural way to define the class of names that will become our intermediate model. And for this we need the following definition.
\begin{definition}
We say that $\dot x$ is \textit{$\cF_\delta$-respected} if there is a support $\vec H$ and a pre-dense $D\subseteq\PP_\delta$, such that for every $p\in D$, and whenever $\gaut{\vec\pi}\in\vec H(p)$, $p\forces_\delta\gaut{\vec\pi}\dot x=\dot x$. We say that $\dot x$ is \textit{hereditarily $\cF_\delta$-respected} if it is $\cF_\delta$-respected, and every $\dot y$ which appears in $\dot x$ is also hereditarily $\cF_\delta$-respected.
\end{definition}

\subsection{Tenacity}
We digress momentarily from the discussion above to introduce a new notion of a symmetric system, one that will help us simplify a lot of the definitions.
\begin{definition}
Let $\tup{\PP,\sG,\sF}$ be a symmetric system. A condition $p\in\PP$ is called \textit{$\sF$-tenacious} if there exists $H\in\sF$ such that for every $\pi\in H$, $\pi p=p$. We say that the forcing $\PP$ is \textit{$\sF$-tenacious} if there is a dense subset of $\sF$-tenacious conditions.
\end{definition}
We will often write that $\tup{\PP,\sG,\sF}$ is a tenacious system, or just that $\PP$ is tenacious if the symmetric system is understood from the context.
\begin{example}\label{ex:cohen}
The easiest example of a tenacious system is Cohen's first model. We force with finite partial functions $p\colon\omega\times\omega\to\{0,1\}$ ordered by reverse inclusion; the automorphisms are induced by finitary automorphisms of $\omega$, $\pi p(\pi n,m)=p(n,m)$; and $\sF$ is the filter generated by $\{\fix(E)\mid E\in[\omega]^{<\omega}\}$ where $\fix(E)$ is the pointwise stabilizer of $E$.

Now given any $p\in\PP$, its domain is finite, so taking $E$ to be the projection of $\dom p$ onto the left coordinate gives us $\fix(E)$ witnessing the tenacity of $\PP$.
\smallskip

On the other hand, if we take the same example, but $\sF$ is the trivial filter, $\{\sG\}$, then only the maximum condition is tenacious.
\end{example}
\begin{remark}
With Yair Hayut we have shown that every symmetric system is equivalent to a tenacious one. The details can be found in \autoref{section:appendix}. This, however, does not minimize the following theorems. They are about choosing a presentation of the forcings in order to simplify the treatment of the iteration, much like the presentation of the iteration as a Kunen- or Jech-style iteration has some impact on how easily we can discuss the symmetric iteration.
\end{remark}
\begin{theorem}\label{thm:iterated tenacity}
Suppose that $\tup{\dot\QQ_\alpha,\dot\sG_\alpha,\dot\sF_\alpha\mid\alpha<\delta}$ is an iteration of symmetric systems such that the two additional conditions hold:
\begin{enumerate}
\item For every $\alpha$, $\dot\QQ_\alpha,\dot\sG_\alpha$ and $\dot\sF_\alpha$ are hereditarily $\cF_\alpha$-respected.
\item For every $\alpha$, $1_\alpha\forces_\alpha\dot\QQ_\alpha\text{ is }\dot\sF_\alpha\text{-tenacious}$.
\end{enumerate}
Then there exists a dense subset $D\subseteq\PP_\delta$ such that for every $p\in D$ there is an excellent support $\vec H$ such that whenever $p\forces_\delta\vec\pi\in\vec H$, we have $\gaut{\vec\pi}p=p$. Therefore, every $\gaut{\vec\pi}\in\vec H(p)$ will satisfy $\gaut{\vec\pi}p=p$.
\end{theorem}
We will say that $p$ satisfying the conclusion of the theorem is \textit{$\cF_\delta$-tenacious}, or just tenacious if $\cF_\delta$ is clear from context, and we will say that \textit{$\vec H$ is a witness for $p$'s tenacity} if it is an excellent support as in the conclusion of the theorem.

Before the proof we will prove the following equivalence, which will be useful in the proof of the theorem and all the proofs to come. So while the definition of being respected is one, we will generally opt for the equivalent definition. It is easier to work with, because tenacious conditions are stable (under ``most'' automorphisms), and stability is important.
\begin{proposition}
Under the conclusion of \autoref{thm:iterated tenacity}, $\dot x$ is $\cF_\delta$-respected if and only if there is a pre-dense subset $D\subseteq\PP_\delta$, such that for every $p\in D$, $p$ there is an excellent support $\vec H_p$ and whenever $\gaut{\vec\pi}\in\vec H_p(p)$, $p\forces_\delta\gaut{\vec\pi}\dot x=\dot x$.
\end{proposition}
\begin{proof}
In the one direction, if $\vec H$ is a support witnessing that $\dot x$ is $\cF_\delta$-respected and $D$ is a pre-dense set of conditions computing it correctly, let $D'$ be a pre-dense set obtained by refining each $p\in D$ to stronger conditions, $q\leq_\delta p$, such that $q$ is tenacious and there is $\vec H_q$, $q\forces_\delta\vec H=\vec H_q$. By shrinking, if necessary, we can assume that $\vec H_q$ is witnessing the tenacity of $q$ (the support mixed from the $\vec H_q$'s could be smaller than $\vec H$, but it is still a support nonetheless). Now if $q\forces_\delta\vec\pi\in\vec H_q$ and $p\in D$ such that $q\leq_\delta p$, then there is some $\vec\sigma$ such that $q\forces_\delta\vec\pi=\vec\sigma$ and $p\forces_\delta\vec\sigma\in\vec H$. Therefore $p\forces_\delta\gaut{\vec\sigma}\dot x=\dot x$, and by the tenacity of $q$ and \autoref{prop:preserving names} we get that $q\forces_\delta\gaut{\vec\pi}\dot x=\gaut{\vec\sigma}\dot x=\dot x$ as well; by induction and \autoref{cor:tenacious composition} this extends to every $\gaut{\vec\pi}\in\vec H_q(q)$.

In the other direction, as we can just mix all the supports to obtain $\vec H$, and from \autoref{prop:correctness} the result follows.
\end{proof}

\begin{proof}[Proof of \autoref{thm:iterated tenacity}]
We prove the theorem by induction on $\delta$. For $\delta=0$ there is nothing to prove, and for $\delta=1$ this is just the definition of tenacity. If $\delta$ is a limit ordinal, if $\bar p\in\PP_\delta$ then there is some $\alpha<\delta$ such that $\bar p\in\PP_\alpha$, and by the induction hypothesis there is some $p\leq_\alpha\bar p$ and an excellent $\cF_\alpha$-support $\vec H$ with the desired properties for $\alpha$; but $p\leq_\delta\bar p$ and extending $\vec H$ by adding $\dot\sG_\beta$ for $\beta\geq\alpha$ preserves the excellence of $\vec H$ and fixing $p$. So it remains to prove the statement for $\delta=\alpha+1$.

Let $\bar p$ be a condition in $\PP_\delta$, we will show that there exists $p\leq_\delta\bar p$ with the wanted properties. By the induction hypothesis there is some $p'\in\PP_\alpha$ such that $p'\leq_\alpha\bar p\restriction\alpha$, and $p'$ lies in the dense set guaranteed by the induction hypothesis for $\PP_\alpha$, with $\vec E$ an excellent $\cF_\alpha$-support witnessing that. By the assumption (2) we may also assume that there exists some $\dot q_\alpha$ which appears in $\dot\QQ_\alpha$ and $\dot H_\alpha$ which appears in $\dot\sF_\alpha$ such that:
\begin{enumerate}
\item $p'\forces_\alpha\dot q_\alpha\leq_{\dot\QQ_\alpha}\bar p(\alpha)$.
\item $p'\forces_\alpha\forall\dot\pi_\alpha\in\dot H_\alpha, \dot\pi_\alpha(\dot q_\alpha)=\dot q_\alpha$.
\item As both $\dot q_\alpha$ and $\dot H_\alpha$ appear in $\dot\QQ_\alpha$ and $\dot\sF_\alpha$ respectively, they are themselves $\cF_\alpha$-respected, and so we may assume that there is an excellent $\cF_\alpha$-support $\vec E'$ such that whenever $\gaut{\vec\pi}\in\vec E'(p')$, $p'\forces_\alpha\gaut{\vec\pi}\dot q_\alpha=\dot q_\alpha$ and $\gaut{\vec\pi}\dot H_\alpha=\dot H_\alpha$.
\item Moreover, as the intersection of two excellent supports is excellent, we can even assume that $\vec E=\vec E'$.
\end{enumerate}
We can do these things by using \autoref{lemma:pre-dense refinement} and \autoref{cor:pairing}. There is no circularity as the assumption is that $\PP_\alpha$ already satisfies the conclusion of the theorem, namely there is a dense subset as wanted.

Let $p$ be such that $p\restriction\alpha=p'$ and $p'\forces_\alpha p(\alpha)=\dot q_\alpha$, and whenever $r\in\PP_\alpha$ is incompatible with $p'$, $r\forces_\alpha p(\alpha)=1_{\dot\QQ_\alpha}$. And similarly extend $\dot H_\alpha$ to a name such that whenever $r$ is incompatible with $p'$, $r\forces_\alpha\dot H_\alpha=\dot\sG_\alpha$.

Now we claim that $\vec H=\vec E^\smallfrown\dot H_\alpha$ is an excellent support for $p$. First note that it is an excellent support since $\vec E$ was excellent and we only added a $\PP_\alpha$-name for the $\alpha$th group. Next, if $p\forces_\delta\vec\pi\in\vec H$, then $p'\forces\vec\pi\restriction\alpha\in\vec E$ and $p'\forces_\alpha\dot\pi_\alpha\in\dot H_\alpha$. Therefore
\[\gaut{\vec\pi}p=\gaut{\vec\pi\restriction\alpha} p'^\smallfrown\gaut{\vec\pi\restriction\alpha}\dot\pi_\alpha(\dot q_\alpha)=p'^\smallfrown\dot q,\]
but $p'\forces_\alpha\dot\pi_\alpha(\dot q_\alpha)=\dot q_\alpha$, and therefore $\gaut{\vec\pi\restriction\alpha}p'\forces_\alpha\gaut{\vec\pi\restriction\alpha}\dot\pi_\alpha(\dot q_\alpha)=\gaut{\vec\pi\restriction\alpha}\dot q_\alpha$.

Using the assumptions we have, this simplifies to $p'\forces_\alpha\dot q=\gaut{\vec\pi}\dot\pi_\alpha(\dot q_\alpha)=\dot q_\alpha$. So $p'\forces_\alpha\dot q=\dot q_\alpha$, and therefore $\gaut{\vec\pi}p=p$ as wanted.
\end{proof}

\begin{lemma}\label{lemma:pre-dense refinement}
Under the conclusion of \autoref{thm:iterated tenacity}, if $\dot x$ is $\cF_\delta$-respected, then we can find a pre-dense set $D$ that for every $p\in D$, $p$ is tenacious and $\vec H_p$ is an excellent support witnessing both this tenacity and that $\dot x$ is respected. Moreover, for any fixed $p\in\PP_\delta$ we can require that some $q\in D$ will satisfy $q\leq_\delta p$.
\end{lemma}
\begin{proof}
As $\dot x$ is respected, there is some $D'$ pre-dense and $\vec H_q$ for $q\in D'$ witnessing this. Below each $q\in D'$, let $D_q$ be a maximal antichain of tenacious conditions and for every $p\in D_q$ take $\vec H_p$ to be a subsupport of $\vec H_q$ witnessing the tenacity of $p$.

Now we claim that $D=\bigcup_{q\in D'}D_q$ is a pre-dense set as wanted. If $p\in D$ and $p\forces_\delta\vec\pi\in\vec H_p$, then there is some $q\in D'$ such that $p\in D_q$, and by \autoref{prop:correctness} there is some $\vec\sigma$ such that $q\forces_\delta\vec\sigma\in\vec H_q$ with $p\forces_\delta\vec\pi=\vec\sigma$. Therefore $q\forces_\delta\gaut{\vec\sigma}\dot x=\dot x$. By the tenacity of $p$ and the fact that $p\forces_\delta\vec\sigma=\vec\pi$ we have that $p\forces_\delta\gaut{\vec\pi}\dot x=\gaut{\vec\sigma}\dot x=\dot x$. The same follows for every $\gaut{\vec\pi}\in\vec H_p(p)$ by induction.

For the last part of the lemma, given any $r\in D$ which is compatible with $p$ we can ensure the maximal antichain below $r$ lies below $p$ as well.
\end{proof}
\begin{corollary}\label{cor:pairing}
Under the conclusion of \autoref{thm:iterated tenacity}, if $\dot x$ and $\dot y$ are $\cF_\delta$-respected names then there exists pre-dense set $D$ witnessing this fact simultaneously.
\end{corollary}
\begin{proof}[Sketch of Proof]
Let $\vec H_x$, $\vec H_y$, $D_x$ and $D_y$ be the supports and pre-dense sets witnessing $\dot x$ and $\dot y$ are $\cF_\delta$-respected, respectively.  Let $D$ be a common refinement of $D_x$ and $D_y$, we may assume that each $p\in D$ is tenacious with witness $\vec H_p$ and we may assume that $\vec H_p\subseteq\vec H_{x,q}\cap\vec H_{y,q'}$ for some $q\in D_x$ and $q'\in D_y$. Then as in the proof of the lemma above, for all $\gaut{\vec\pi}\in\vec H_p(p)$ it holds that $p\forces_\delta\gaut{\vec\pi}\dot x=\dot x$ and $\gaut{\vec\pi}\dot y=\dot y$.
\end{proof}

So what does tenacity give us? Well, despite that it is easy to prove that if $p\forces_\delta\vec\pi=\vec\sigma$ implies that $\gaut{\vec\pi}p=\gaut{\vec\sigma}p$, we could only prove that $p$ itself also forces that $\gaut{\vec\pi}\dot x=\gaut{\vec\sigma}\dot x$ under the assumption that $\gaut{\vec\pi}p=p$. If $p$ is tenacious we can assume that we work with automorphisms which are in the support witnessing that, and not only that, since the witness for the tenacity of $p$ is an excellent support, we can ``pull'' any automorphisms from any extension of $p$ to something that $p$ already forces to be in its tenacity witness. Combining this with the lemma above we get that we can always make the assumption that if $\dot x$ is respected, then there is---generically---a large group which respects it.

\section{Iterated symmetric extensions}\label{section:IS}
\begin{definition}
A \textit{symmetric iteration} (of length $\delta$) is $\tup{\dot\QQ_\alpha,\dot\sG_\alpha,\dot\sF_\alpha\mid\alpha<\delta}$ and $\tup{\PP_\alpha,\cG_\alpha,\cF_\alpha\mid\alpha\leq\delta}$ if the following holds
\begin{enumerate}
\item $\PP_\alpha$ is the finite support iteration of $\dot\QQ_\beta$ for $\beta<\alpha$, $\cG_\alpha$ is the generic semi-direct product of $\dot\sG_\beta$ for $\beta<\alpha$ and $\cF_\alpha$ is the system of supports generated by $\dot\sF_\beta$ for $\beta<\alpha$.
\item For every $\alpha$, $1_\alpha\forces_\alpha$``$\tup{\dot\QQ_\alpha,\dot\sG_\alpha,\dot\sF_\alpha}^\bullet$ is a tenacious system''.\footnote{It might be the case that $\sF_\alpha$ is not a filter in the generic extension, only in the ``guessed'' intermediate extension. This can be fixed after defining the $\forces^\IS$ relation, and by noticing that in either case it is a normal filter base.}
\item For every $\alpha$, the names $\dot\QQ_\alpha,\dot\sG_\alpha$ and $\dot\sF_\alpha$ are hereditarily $\cF_\alpha$-respected.
\item For every $\alpha$, the names $\dot\QQ_\alpha,\dot\sG_\alpha$ and $\dot\sF_\alpha$ are respected by every automorphism in $\cG_\alpha$.
\end{enumerate}
We denote by $\IS_\alpha$, for $\alpha\leq\delta$, the class of $\PP_\alpha$-names which are hereditarily $\cF_\alpha$-respected.
\end{definition}
The $\IS_\delta$ class will be interpreted as our intermediate model. The idea is that $\IS_1$ will be interpreted as the symmetric extension obtained by the first step; the second step $\tup{\dot\QQ_1,\dot\sG_1,\dot\sF_1}$ lies in $\IS_1$, so it is a valid candidate for a symmetric system for the second step, and so on. At limit steps, $\IS_\delta$ is going to be something resembling to ``definable from finitely many elements from $\bigcup_{\alpha<\delta}\IS_\alpha$''.

A natural question to ask now is why did we take the effort to even define $\IS_\alpha$, rather than just talk about $\cF_\alpha$-respected names. The reason is that we do not require that the symmetric systems in our iterations are hereditarily respected at each step of the iteration. So at least in theory it is conceivable that we can find an iteration where each iterand is respected, but not hereditarily respected, and this would be akin to forcing from outside the model. While there might be some interest in such approach, it has inherent difficulties and it does not faithfully reflect what we would expect from an iteration of symmetric extensions.
\smallskip

For the rest of the section we will work in the context of a symmetric iteration of length $\delta$.
\subsection{Properties of \texorpdfstring{$\IS$}{IS} names}
\begin{theorem}\label{thm:IS-names are closed under mixing}
Let $\dot x$ be a $\PP_\delta$-name and suppose that $D\subseteq\PP_\delta$ is a maximal antichain such that for every $p\in D$ there exists $\dot x_p\in\IS_\delta$ such that $p\forces_\delta\dot x=\dot x_p$. Then $\dot x\in\IS_\delta$.
\end{theorem}
\begin{proof}
We prove this by induction on the rank of $\dot x$; it is enough to prove that $\dot x$ is $\cF_\delta$-respected as the induction hypothesis immediately guarantees in this case that it is going to be in $\IS_\delta$. We can assume without loss of generality that each $p\in D$ satisfies the following properties:
\begin{enumerate}
\item $p$ lies in the pre-dense open set witnessing the fact $\dot x_p$ is $\cF_\delta$-respected with $\vec H_p$ as a support, and
\item $p$ is tenacious, and $\vec H_p$ is an excellent support witnessing this as well.
\end{enumerate}
If these conditions fail to hold, we can refine $D$ by the definition of symmetric iterations.

We claim that $D$ witnesses that $\dot x\in\IS_\delta$. If $p\in D$ and $\gaut{\vec\pi}\in\vec H_p(p)$, then \[p\forces_\delta\gaut{\vec\pi}\dot x=\gaut{\vec\pi}\dot x_p=\dot x_p=\dot x\] (the first equality comes from the fact that $\gaut{\vec\pi}p=p$ and the Symmetry Lemma).
\end{proof}
We get from the above theorem that if we treat an automorphism sequence $\vec\pi$ as a name of a sequence (namely, $\{\tup{\check\alpha,\dot\pi_\alpha}^\bullet\mid\alpha<\delta\}^\bullet$) then $\vec\pi\in\IS_\delta$, and similarly this is the case for any support $\vec H$.

The following proposition follows immediately from \autoref{cor:pairing}
\begin{proposition}\label{prop:pairing}
If $\dot x,\dot y\in\IS_\delta$, then $\{\dot x,\dot y\}^\bullet\in\IS_\delta$.\qed
\end{proposition}

\begin{lemma}
Suppose that $\vec\pi$ is an automorphism sequence, then there is a dense set of tenacious conditions $p$, such that $\gaut{\vec\pi}p$ is also tenacious.
\end{lemma}
\begin{proof}
Fix $D$ to be a pre-dense set witnessing that $\vec\pi\in\IS_\delta$. We may assume, using \autoref{lemma:pre-dense refinement} that every $q\in D$ is tenacious and $\vec H_q$ is a witness for the tenacity of $q$ and for the fact that $\vec\pi\in\IS_\delta$.

Now we claim that if $p\leq_\delta q$ for a tenacious $p$ and $q\in D$, then $\gaut{\vec\pi}p$ is also tenacious. First observe that if $\gaut{\vec\sigma}p=p$, then $\gaut{\vec\pi}\gaut{\vec\sigma}\gaut{\vec\pi^{-1}}\left(\gaut{\vec\pi}p\right)=\gaut{\vec\pi}p$. So the first attempt would be the conjugate $\vec H_p$ by $\gaut{\vec\pi}$ in order to obtain a support for $\gaut{\vec\pi}p$.

However, recall the equation \ref{eqn:conjugation} at the beginning of the previous section:
\[\vec\pi\circ\vec\sigma\circ\vec\pi^{-1}=\gaut{\vec\pi}\tup{(\gaut{\vec\sigma^{-1}}\dot\pi_\alpha)\dot\sigma_\alpha\dot\pi_\alpha^{-1}\mid\alpha<\delta}.\]
So conjugation is not going to be enough. We also want to get rid of that $\gaut{\vec\sigma^{-1}}$, and to do so we claim that $\vec E$ defined by $\gaut{\vec\pi}\tup{\dot\pi_\alpha(\dot H_{p,\alpha}\cap H_{q,\alpha})\dot\pi_\alpha^{-1}\mid\alpha<\delta}$ is going to witness the tenacity of $\gaut{\vec\pi}p$.

First we observe that the intersection of two excellent supports is an excellent support, each $\dot\sF_\alpha$ is guaranteed to be a normal filter, and it is respected by $\gaut{\vec\pi}$. So $\vec E$ is indeed an excellent support. Now suppose that $\gaut{\vec\pi}p\forces_\delta\vec\tau\in\vec E$, we claim that $\gaut{\vec\tau}\gaut{\vec\pi}p=\gaut{\vec\pi}p$. 

Apply $\gaut{\vec\pi^{-1}}$ to $\gaut{\vec\pi}p\forces_\delta\vec\tau\in\vec E$, by the \nameref{lemma:symmetry} and considering one coordinate, we get that $p\forces_\delta\gaut{\vec\pi^{-1}}\dot\tau_\alpha=\dot\pi_\alpha\dot\sigma_\alpha\dot\pi_\alpha^{-1}$, where $\dot\sigma_\alpha$ is some automorphism such that $p\forces_\delta\dot\sigma_\alpha\in\dot H_{p,\alpha}\cap\dot H_{q,\alpha}$.

It follows that there is some $\vec\sigma'$ such that $q\forces_\delta\vec\sigma'\in\vec H_q$ and $p\forces_\delta\vec\sigma=\vec\sigma'$. Therefore $\gaut{\vec\sigma}p=p$ and $p\forces\gaut{\vec\sigma}\dot\pi_\alpha=\gaut{\vec\sigma'}\dot\pi_\alpha=\dot\pi_\alpha$. This means that $p\forces_\delta\gaut{\vec\pi^{-1}}\vec\tau=\gaut{\vec\pi^{-1}}(\vec\pi\circ\vec\sigma\circ\vec\pi^{-1})$. Therefore, $\gaut{\vec\pi}p\forces_\delta\vec\tau=\vec\pi\circ\vec\sigma\circ\vec\pi^{-1}$, which gives us
\[\gaut{\vec\tau}\gaut{\vec\pi}p=\gaut{\vec\pi}\gaut{\vec\sigma}\gaut{\vec\pi^{-1}}\gaut{\vec\pi}p=\gaut{\vec\pi}p.\qedhere\]
\end{proof}
This idea is essentially telling us that the supports are ``generically normal''. We may have to shrink something down, or extend some condition, but on a dense set things work out. And this brings us to the following important theorem.
\begin{theorem}\label{thm:IS-names are closed under aut}
If $\gaut{\vec\pi}\in\cG_\delta$ and $\dot x\in\IS_\delta$, then $\gaut{\vec\pi}\dot x\in\IS_\delta$.
\end{theorem}
\begin{proof}
We prove this by induction on $\rank(\dot x)$. It suffices to prove that $\gaut{\vec\pi}\dot x$ is respected, as the induction hypothesis ensures that every $\dot y$ which appears in $\dot x$ will satisfy that $\gaut{\vec\pi}\dot y\in\IS_\delta$ also. Let $D$ be a pre-dense set witnessing that $\dot x,\vec\pi\in\IS_\delta$.

We may also assume that $D$ is a pre-dense set such that whenever $p\in D$, both $p$ and $\gaut{\vec\pi}p$ are tenacious, and that $\vec H_p$ is witnessing both the tenacity of $p$, as well the fact that $\dot x,\vec\pi\in\IS_\delta$. We work below $p$, and follow the same path as the previous proof: define $\vec E_p=\gaut{\vec\pi}\tup{\dot\pi_\alpha\dot H_{p,\alpha}\dot\pi^{-1}_\alpha\mid\alpha<\delta}$, and we claim that $D'=\{\gaut{\vec\pi}p\mid p\in D\}$ is a pre-dense set witnessing that $\gaut{\vec\pi}\dot x\in\IS_\delta$ with $\vec E_p$ as the support for $\gaut{\vec\pi}p$.

Suppose that $\gaut{\vec\pi}p\in D'$, and $\gaut{\vec\pi}p\forces_\delta\vec\tau\in\vec E_p$, then from the proof of the lemma we get that $p\forces_\delta\gaut{\vec\pi^{-1}}\vec\tau=\gaut{\vec\pi^{-1}}\vec\pi\circ\vec\sigma\circ\vec\pi^{-1}$, where $\vec\sigma$ is an automorphism sequence such that $p\forces_\delta\vec\sigma\in\vec H_p$ and therefore $p\forces_\delta\gaut{\vec\sigma}\dot x=\dot x$. It follows that $\gaut{\vec\pi}p\forces_\delta\vec\tau=\vec\pi\circ\vec\sigma\circ\vec\pi^{-1}$, and that $\gaut{\vec\tau}\gaut{\vec\pi}p=\gaut{\vec\pi}p$. Therefore $\gaut{\vec\pi}p\forces_\delta\gaut{\vec\tau}\gaut{\vec\pi}\dot x=\gaut{\vec\pi}\gaut{\vec\sigma}\gaut{\vec\pi^{-1}}\gaut{\vec\pi}\dot x$. As before, it follows that $\gaut{\vec\pi}p\forces\gaut{\vec\tau}\gaut{\vec\pi}\dot x=\gaut{\vec\pi}\dot x$.

Now proceed by induction to conclude the same for every $\gaut{\vec\tau}\in\vec E_p(\gaut{\vec\pi}p)$.
\end{proof}
\begin{theorem}\label{thm:ZF models}
Suppose that $G\subseteq\PP_\delta$ is a $V$-generic filter. Then the class $M$ defined as $\IS_\delta^G=\{\dot x^G\mid\dot x\in\IS_\delta\}$ satisfies $\ZF$ and $V\subseteq M\subseteq V[G]$. Moreover, for every $\alpha<\delta,$ $\IS_\alpha^{G\restriction\alpha}\subseteq M$.
\end{theorem}
\begin{proof}[Sketch of Proof]
The proof is very similar to the standard proof about symmetric extensions, e.g.\ the one in \cite{Jech:AC}. By the hereditary definition of $\IS_\delta$, $M$ is a transitive class. The $\check x$ names are fixed under all automorphisms, so they are certainly in $\IS_\delta$, so $V\subseteq M\subseteq V[G]$. By \autoref{thm:IS-names are closed under aut}, $\{\dot x\in\IS_\delta\mid\rank(\dot x)<\eta\}^\bullet$ is stable under all automorphisms in $\cG_\delta$, and therefore it is in $\IS_\delta$, and it is a witness to the fact that $M$ is almost-universal.\footnote{A class $M$ is almost universal if for every set $x\subseteq M$, there is $y\in M$ such that $x\subseteq y$.} Usual arguments also give us that $M$ is closed under G\"odel operations, and therefore it is a model of $\ZF$.

The additional part is trivial, as $\IS_\alpha\subseteq\IS_\delta$ for $\alpha<\delta$ (every $\PP_\alpha$-name is a $\PP_\delta$-name; and simply extending the $\cF_\alpha$-support in the obvious way, by concatenating $\dot\sG_\gamma$ for $\gamma\in[\alpha,\delta)$, gives the wanted conclusion).
\end{proof}
\subsection{The \texorpdfstring{$\IS$}{IS}-forcing relation}
Now that we have isolated a class of names which predict the intermediate model, we would like to also have a forcing relation so that we can predict facts about these names from the ground model. In this section we are not particularly interested in the iteration itself, so we will omit the $\delta$ subscript from where it is not needed: $\IS_\delta$ becomes $\IS$ and $\forces_\delta$ becomes $\forces$, and so on.
\begin{definition}[The Forcing Relation] We define $p\forces^\IS\varphi$ by induction on the complexity of $\varphi$.
\begin{itemize}
\item $p\forces^\IS\dot x\in\dot y$ if and only if $p\forces\dot x\in\dot y$ and $\dot x,\dot y\in\IS$.
\item $p\forces^\IS\dot x=\dot y$ if and only if $p\forces\dot x=\dot y$ and $\dot x,\dot y\in\IS$.
\item $p\forces^\IS\varphi\land\psi$ if and only if $p\forces^\IS\varphi$ and $p\forces^\IS\psi$.
\item $p\forces^\IS\lnot\varphi$ if and only if there is no $q\leq p$ such that $q\forces^\IS\varphi$.
\item $p\forces^\IS\exists x\varphi(x)$ if and only if $\{q\mid\exists\dot x\in\IS:q\forces^\IS\varphi(\dot x)\}$ is dense below $p$.
\end{itemize}
\end{definition}
The keen-eyed observer will note that this is really just the relativization of the usual forcing relation to the class of names $\IS$. We write a list of useful properties of the $\forces^\IS$ relation, they either follow immediately from the definition or from the same properties for the $\forces$ relation.
\begin{proposition}[Substitution Lemma]
If $p\forces^\IS\varphi(\dot x_1,\dots,\dot x_n)\land\bigwedge_{i=1}^n\dot x_i=\dot y_i$, then $p\forces^\IS\varphi(\dot y_1,\ldots,\dot y_n).$\qed
\end{proposition}
\begin{proposition}\hfill
\begin{enumerate}
\item If $p\forces^\IS\varphi$ and $q\leq p$, then $q\forces^\IS\varphi$.
\item There is no $p$ such that $p\forces^\IS\varphi\land\lnot\varphi$.
\item $p\forces^\IS\varphi$ if and only if $p\forces^\IS\lnot\lnot\varphi$.
\item $p\forces^\IS\varphi$ if and only if there is no $q\leq p$ such that $q\forces^\IS\lnot\varphi$.
\item $p\forces^\IS\varphi$ if and only if the set $\{q\mid q\forces^\IS\varphi\}$ is dense below $p$.
\item $p\forces^\IS\exists x\varphi(x)$ if and only if there exists $\dot x\in\IS$ such that $p\forces^\IS\varphi(\dot x)$.\footnote{This makes use of mixing, of course.}\qed
\end{enumerate}
\end{proposition}
\begin{lemma}[The ($\IS$-)Symmetry Lemma] If $\gaut{\vec\pi}\in\cG$ and $\dot x\in\IS$, then \[p\forces^\IS\varphi(\dot x)\iff\gaut{\vec\pi}p\forces^\IS\varphi(\gaut{\vec\pi}\dot x).\qed\]
\end{lemma}
\begin{theorem}[The Forcing Theorem for $\forces^\IS$]
The following are equivalent:
\begin{enumerate}
\item $p\forces^\IS\varphi$.
\item For every $V$-generic $G$ such that $p\in G$, $\IS^G\models\varphi$.\qed
\end{enumerate}
\end{theorem}

\section{Toy example: No free ultrafilters on \texorpdfstring{$\omega$}{w}}\label{section:toy}
We digress from the rest of the general theory in order to provide a small example of using the construction we have so far. This might not be a very exciting example, and it is quite easy to achieve with symmetric extensions in the usual way. However, they can provide some basic ground for understanding the mechanism of symmetric iterations.

Our setting is very simple. We are going to force with an iteration of length $\omega$ and for every $n<\omega$ we use the same symmetric system:
\begin{enumerate}
\item $\dot\QQ_n$ is the Cohen forcing, presented as finite binary sequences;
\item $\dot\sG_n$ will be the name of the group $\{\pi_A\mid A\subseteq\omega\}$, in the $n$th extension, where $\pi_A p(n)=p(n)$ if $n\in A$, and $\pi_A p(n)=1-p(n)$ otherwise.
\item $\dot\sF_n$ will be the improper filter.
\end{enumerate}
We will omit $\omega$ from the subscripts, so $\PP=\PP_\omega$, $\IS=\IS_\omega$ and so on. Let $G$ be a $V$-generic filter for $\PP$, and let $N=\IS^G$. Note that the last condition means that every $\PP_n$-name is in $\IS$. 
\begin{proposition}
$N\models$ Every ultrafilter on $\omega$ is principal.
\end{proposition}
\begin{proof}
Suppose that $\dot U\in\IS$ is such that $1\forces^\IS\dot U$ is an ultrafilter on $\check\omega$. Let $D$ be a pre-dense set witnessing that $\dot U\in\IS$, and suppose that $p\in D$, and $\vec H_p$ is an excellent support such that $p\forces\gaut{\vec\pi}\dot U=\dot U$ for every $\gaut{\vec\pi}\in\vec H_p(p)$. We will show that every $p\in D$ forces that $\dot U$ is principal.

Let $n>\max\supp p+\max C(\vec H_p)$, and let $\dot c_n$ be the canonical name for the Cohen real added by $\dot\QQ_n$, namely $\{\tup{p,\check m}\mid \supp p=\{n\}\land p(n,m)=1\}$. Suppose that $q\leq p$ and $q\forces\dot c_n\in\dot U$. We claim that $q\restriction n+1\forces\dot c_n\in\dot U$: for every $k$, $\dot\sG_k$ witnesses the homogeneity of $\dot\QQ_k$, so the entire iteration is weakly homogeneous; taking any two extensions of $q\restriction n+1$, we can make them compatible without moving anything below $n+1$ and such automorphism $\vec\sigma$ will satisfy that $\gaut{\vec\sigma}\in\vec H_p(p)$, as well $\gaut{\vec\sigma}\dot c_n=\dot c_n$, so both extensions must agree on the truth value of $\dot c_n\in\dot U$. So we may assume that $q=q\restriction n+1$.

Let $\dot A$ be a $\PP_n$-name for a subset of $\omega$ such that $1_n\forces_n\dom q(n)=\dot A$, and let $\dot\pi$ be the $\PP_n$-name for the automorphism induced on $\dot\QQ_n$ by $\dot A$. Then the following are true,
\begin{enumerate}
\item $1_n\forces_n\dot\pi q(n)=q(n)$, and consequently $\gaut{\dot\pi}q=q$ and $\gaut{\dot\pi}\in\vec H_p(p)$.
\item $1_n\forces_n\dot\pi\dot c_n\cap\dot c_n=q(n)^{-1}(\check 1)$.
\end{enumerate}
Therefore, $q\forces\gaut{\dot\pi}\dot U=\dot U$, but also $q\forces q(n)^{-1}(\check 1)\in\dot U$. Therefore $q\forces``$there is a finite set in $\dot U$''. It follows that if $p\forces^\IS\dot U$ is an ultrafilter on $\check\omega$, then no extension of $p$ can force that $\dot U$ is free.
\end{proof}
The construction is originally due to Solomon Feferman (\cite[Theorem~4.12]{Feferman:1964} for the original result, or \cite[Example~15.59]{Jech:ST} for a modern treatment).

\section{Symmetric iterations are iterations of symmetric extensions}\label{section:justification}
\subsection{Outline of the section}
We would want to justify the name ``symmetric iterations'' and to prove that if $G$ is a $V$-generic filter for $\PP_\delta$, then for every $\alpha<\delta$, $\IS_{\alpha+1}^{G\restriction\alpha+1}$ is a symmetric extension of $\IS_\alpha^{G\restriction\alpha}$ using the symmetric system $\tup{\dot\QQ_\alpha^{G\restriction\alpha},\dot\sG_\alpha^{G\restriction\alpha},\dot\sF_\alpha^{G\restriction\alpha}}$.

Ideally, would we would like to have some sort of a decomposition theorem at play. Something similar to the usual case of a finite support iteration: $\PP_\delta$ is isomorphic to the iteration of $\PP_\alpha\ast(\dot\QQ_\alpha\ast\PP_\delta/\PP_{\alpha+1})$, as well as to the iteration $(\PP_\alpha\ast\dot\QQ_\alpha)\ast\PP_\delta/\PP_{\alpha+1}$. But here we have a problem. When we want to decompose the automorphisms and the supports, we relied on mixing so heavily, that there is no nice way of doing so in general, as we may have mixed over antichains that were unstable under most automorphisms we had at our disposal.

However, not all is lost. For ``sufficiently cooperative'' $\PP_{\alpha+1}$-names we actually have a good decomposition, and as luck would have it, every name in $\IS_{\alpha+1}$ is eventually forced to be equal to such a ``cooperative'' name. So the plan is to first look at two-step iterations, and see how the automorphisms and supports behave when we move from a $\PP\ast\dot\QQ$-name to a $\PP$-name for a $\dot\QQ$-name. This will help us to formulate the two main lemmas for going back and forth $\IS_\alpha$ and $\IS_{\alpha+1}$, as well $V[G\restriction\alpha]$ and $\IS_\alpha^{G\restriction\alpha}$. From the lemmas we will have the two theorems that a symmetric iteration is indeed an iteration of symmetric extensions, and that if we take a symmetric extension of a symmetric iteration, then we could have done that by extending the original symmetric iteration by exactly that last step (although we may have to pay a price by shrinking our automorphism groups).
\subsection{The factorization lemmas}
\begin{definition}
Let $\dot x$ be a $\PP\ast\dot\QQ$-name. We denote by $[\dot x]$ the $\PP$-name of the $\dot\QQ$-name for $\dot x$, defined recursively by
\[[\dot x]=\{\tup{p,\tup{\dot q,[\dot y]}^\bullet}\mid\tup{\tup{p,\dot q},\dot y}\in\dot x\}.\]
\end{definition}
Here are two general lemmas which can be proved by induction on the rank of $\PP\ast\dot\QQ$-names.
\begin{lemma}\label{lemma:factoring and equality}
Suppose that $\tup{p,\dot q}\forces_{\PP\ast\dot\QQ}\dot x=\dot x'$, then $p\forces_\PP``\dot q\forces_{\dot\QQ}[\dot x]=[\dot x']"$.\qed
\end{lemma}
\begin{lemma}\label{lemma:factoring and aut}
Suppose that $\pi\in\aut(\PP)$ respects $\dot\QQ$ and $1_\PP\forces_\PP\dot\sigma\in\aut(\dot\QQ)$. Then \[1_\PP\forces_\PP[\gaut{\tup{\pi,\dot\sigma}}\dot x]=\pi(\dot\sigma[\dot x]).\]Alternatively, $1_\PP\forces_\PP[\gaut{\pi}\dot x]=\pi[\dot x]$ and $[\gaut{\dot\sigma}\dot x]=\dot\sigma[\dot x]$.\qed
\end{lemma}
\begin{definition}
Suppose that $\PP\ast\dot\QQ$ is a two-step iteration of symmetric extensions. We denote by $\HS^\bullet_{\sF_\QQ}$ the class of $\PP$-names which are forced to be hereditarily $\dot\sF_\QQ$-symmetric $\dot\QQ$-names.
\end{definition}
Morally, when we are iterating two symmetric extensions, we want to have $\dot x\in\IS$ if and only if ``$[\dot x]\in\HS_{\sF_\PP}$ and $1_\PP\forces_\PP[\dot x]\in\HS_{\sF_\QQ}^\bullet$''. In the general context of symmetric iteration, if $\PP_\alpha\ast\dot\QQ_\alpha$ is a successor length symmetric iteration, we would have liked to have $\dot x\in\IS_{\alpha+1}$ if and only if $[\dot x]\in\IS_\alpha$ and $1_\alpha\forces_\alpha^\IS[\dot x]\in\HS_{\sF_\alpha}^\bullet$. Unfortunately, this is not going to be the case, since $\dot x$ might have been the result of mixing over antichains which are ``invisible'' to the intermediate extension. Namely, if $G$ is a generic for $\PP_\alpha$, then $\dot x$ might be the result of mixing over names from $\IS_\alpha^G$ using an antichain of $\dot\QQ_\alpha^G$ which lies outside $\IS_\alpha^G$.

This shows that mixing, while very useful and even necessary to a certain degree, is an obstruction. However, it is the only obstruction, as shown below.
\begin{lemma}[The First Factorization Lemma]
Let $\PP_\alpha\ast\dot\QQ_\alpha$ be a successor length symmetric iteration and $G$ is a $V$-generic filter for $\PP_\alpha$. If $\dot x\in\IS_{\alpha+1}$, then in $V[G]$ there exists a maximal antichain $D\subseteq\QQ_\alpha$ such that for every $q\in D$ there exists $\dot x_q\in\IS_\alpha^G$ which is a hereditarily $\cF_\alpha$-symmetric $\QQ_\alpha$-name and $q\forces_{\QQ_\alpha}[\dot x]^G=\dot x_q$.
\end{lemma}
An easy corollary of this lemma is that if we take a one-step symmetric iteration, namely a symmetric extension, then $\IS$ the closure of $\HS$ under mixing.
\begin{proof}
To ease the reading we omit $\alpha$ from the subscript wherever it appears. The proof is by induction on the rank of $\dot x$. Let $\dot x\in\IS_{\alpha+1}$. We may assume the following,
\begin{enumerate}
\item Every condition $p$ which appears in $\dot x$ is such that $p\restriction\alpha\forces p(\alpha)=\dot q$ for some $\dot q$ which appears in $\dot\QQ$. 
\item Every $\dot y$ which appears in $\dot x$ is such that 
\begin{enumerate}
\item $[\dot y]\in\IS_\alpha$, and
\item $1\forces^\IS[\dot y]\in\HS_{\cF}^\bullet$.
\end{enumerate}
\end{enumerate}
If $\dot x$ does not satisfy these properties, we can replace it by such a name which is guaranteed to be equal in $V[G]$ to $\dot x$. If we had to replace the name it is possible that we changed its rank too, but this turns to have no consequence on the rest of the proof.\footnote{It is also possible to prove that we can always find a suitable replacement name which does not increase the rank, but as said, it is of no consequence.}

Let $p^\smallfrown\dot q$ be a condition which is tenacious in $\PP\ast\dot\QQ$ as witnessed by an excellent support $\vec H_p$, $p^\smallfrown\dot q$ lies in the pre-dense set witnessing that $\dot x\in\IS_{\alpha+1}$, and $p\in G$. We can even require that $p$ is in the pre-dense set witnessing the fact $\dot q$ is hereditarily $\cF$-respected with $\vec H\restriction\alpha$ as an $\cF$-support witnessing that. Define the following $\PP\ast\dot\QQ$-name:
\[\dot u=\left\{\tup{\gaut{\vec\pi} r,\gaut{\vec\pi}\dot y}\middd\begin{array}{l}r\leq_{\alpha+1} p^\smallfrown\dot q\\\exists s:r\leq_{\alpha+1}s,\text{ and }\tup{s,\dot y}\in\dot x\\ \gaut{\vec\pi}\in\vec H_p(p)\end{array}\right\}.\]
It might be curious why we wrote $\vec H_p(p)$ rather than $\vec H_p(p^\smallfrown\dot q)$. But due to the excellence of $\vec H_p$, no names involved in $\vec H_p$ are $\PP\ast\dot\QQ$-names; they are all $\PP$-names. So using $p$ suffices for this matter.
To complete the proof we need to prove these four things,
\begin{enumerate}\renewcommand{\theenumi}{\normalfont\textbf{\Alph{enumi}}}
\item $p^\smallfrown\dot q\forces_{\alpha+1}\dot u=\dot x$, or equivalently $p\forces``\dot q\forces_{\dot\QQ}[\dot u]=[\dot x]"$.
\item $\dot u\in\IS_{\alpha+1}$.
\item $[\dot u]\in\IS$.
\item $1\forces^\IS[\dot u]\in\HS_{\sF}^\bullet$.
\end{enumerate}
To prove (\textbf A), the first inclusion $p^\smallfrown\dot q\forces_{\alpha+1}\dot x\subseteq\dot u$ is trivial, simply take $\vec\pi$ to be the identity. In the other direction, suppose that $s\forces_{\alpha+1}\dot t\in\dot u$, then we may assume without loss of generality that $\tup{s,\dot t}\in\dot u$. By definition, $\tup{s,\dot t}=\tup{\gaut{\vec\pi}r,\gaut{\vec\pi}\dot y}$ such that $r\forces_{\alpha+1}\dot y\in\dot x$ and $r\leq_{\alpha+1}p^\smallfrown\dot q$. However, $\gaut{\vec\pi}(p^\smallfrown\dot q)=p^\smallfrown\dot q$ by the tenacity and choice of support $\vec H_p$, and therefore $s\leq_{\alpha+1}p^\smallfrown\dot q$, so $s\forces_{\alpha+1}\gaut{\vec\pi}\dot x=\dot x$ hence \[r\forces_{\alpha+1}\dot y\in\dot x\iff s\forces_{\alpha+1}\dot t\in\dot x.\]

The fact that $\dot u\in\IS_{\alpha+1}$ follows now immediately from the mixing lemma for $\IS$-names, since $1_{\alpha+1}\forces_{\alpha+1}\dot u\in\{\dot x,\check\varnothing\}^\bullet$.

Next we prove that $[\dot u]\in\IS$. We assumed that if $\dot y$ appears in $\dot x$, then $[\dot y]\in\IS$, and therefore it is enough to prove that $[\dot u]$ is $\cF$-respected in order to show it is in $\IS$. Let $D\subseteq\PP$ be a maximal antichain with $p\in D$ and every $p'\in D$ is tenacious with support $\vec E_{p'}$ and $\vec E_p=\vec H_p\restriction\alpha$. We claim that $D$ witnesses $[\dot u]\in\IS$. First, if $p'\in D\setminus\{p\}$, then $p'\forces[\dot u]=\check\varnothing$ so there is nothing to prove there. If $p\forces\vec\pi\in\vec E_p$, then $\gaut{\vec\pi^\smallfrown\id^\bullet}\in H_p(p)$ (as an automorphism of $\PP\ast\dot\QQ$), and the result follows from \autoref{lemma:factoring and aut}.

To prove (\textbf D), it is again enough to work below $p$, as any incompatible condition forces that $[\dot u]$ will be the empty set which is always hereditarily symmetric. Let $\dot H=\dot H_\alpha$. If $p'\leq p$, and $p'\forces\dot\sigma\in\dot H$, then there is some $\dot\pi_\alpha$ such that $p\forces\dot\pi_\alpha\in\dot H$ and $p'\forces\dot\pi_\alpha=\dot\sigma$. Repeating the last argument, this time taking $\gaut{\dot\sigma}$ as our automorphism we obtain that $p\forces\dot\sigma[\dot u]=[\dot u]$ as wanted.

The above construction happens below a specific $q$ in $\QQ$, so to finish the proof we need to find some maximal antichain $D\in V[G]$  where for every $q\in D$ we can find such $\dot x_q$. Of course, given $q$ for which we can find such a name, $[\dot u]$ as defined above works as $\dot x_q$. Moreover, there is a dense open set of $q\in\QQ$ for which the above construction works, so we can refine them to a maximal antichain $D$ as wanted.
\end{proof}

This gives us the first part of the factorization. Namely, if $\dot x\in\IS_{\alpha+1}$, then it is a mixing of names which project to $\IS_\alpha$ and $\HS_{\sF_\alpha}^\bullet$. The second factorization lemma is the converse of that. Luckily, here we can use mixing to argue for a slightly weaker statement.
\begin{lemma}[The Second Factorization Lemma]
Let $\dot x$ be a $\PP_{\alpha+1}$-name such that $[\dot x]\in\IS_\alpha$ and $1_\alpha\forces^\IS_\alpha[\dot x]\in\HS_{\sF_\alpha}^\bullet$. Then $\dot x\in\IS_{\alpha+1}$.
\end{lemma}
\begin{proof}
We prove this by induction on the rank of $\dot x$; it suffices to show that $\dot x$ is $\cF_{\alpha+1}$-respected in order to conclude that $\dot x\in\IS_{\alpha+1}$. Let $D\subseteq\PP_\alpha$ be a pre-dense set which witnesses that $[\dot x]\in\IS_\alpha$, and let $\dot H\in\IS_\alpha$ be a name such that for every $p\in D$, $p\forces_\alpha^\IS\forall\dot\pi\in\dot H:\dot\pi[\dot x]=[\dot x]$. We may also assume that $D$ is a pre-dense set witnessing that $\dot H\in\IS_\alpha$.

We claim that $D$ is also a pre-dense set witnessing that $\dot x\in\IS_{\alpha+1}$. Given $p\in D$, let $\vec H_p$ be a support such that:
\begin{enumerate}
\item $\vec H_p\restriction\alpha$ witnesses that $p$ is tenacious and that $[\dot x],\dot H$ are $\cF_\alpha$-respected.
\item $\vec H_{p,\alpha}=\dot H$.
\end{enumerate}
Then $\vec H_p$ is excellent, and it witnesses that $p$ is $\cF_{\alpha+1}$-tenacious. Suppose that $p\forces_{\alpha+1}\vec\pi\in\vec H_p$, then $p\forces_\alpha[\gaut{\vec\pi}\dot x]=\gaut{\vec\pi\restriction\alpha}\dot\pi_\alpha([\dot x])$. However, $\gaut{\vec\pi\restriction\alpha}p\forces\gaut{\vec\pi\restriction\alpha}\dot\pi_\alpha\in\dot H$, so the conclusion follows.
\end{proof}
\begin{remark}
The above proof may seem a bit odd. We skipped entirely any need to extend the conditions in $D$ to include information from $\dot\QQ_\alpha$. However, the assumption was that $1_\alpha\forces_\alpha[\dot x]\in\HS_{\sF_\alpha}^\bullet$. So we could have deduced the information on $\dot H_\alpha$ already from conditions in $\PP_\alpha$. As $\IS$ is the closure of $\HS$ under mixing, if we had allowed $[\dot x]$ to be equal to a mixing of names from $\HS_{\sF_\alpha}^\bullet$, then we would have needed to include more information from $\dot\QQ_\alpha$.
\end{remark}
\subsection{Two easy theorems}
We finish this section with two easy theorems which are corollaries from the above lemmas. And these tell us that iterating symmetric extensions is indeed the same as doing a symmetric iteration. To simplify the statements of the theorems, let us set a common context here. $\tup{\dot\QQ_\alpha,\dot\sG_\alpha,\dot\sF_\alpha\mid\alpha<\delta}$ is a symmetric iteration of length $\delta$, and let $G$ be a $V$-generic filter for $\PP_\delta$. 

\begin{theorem}
For every $\alpha<\delta$, $\IS_{\alpha+1}^{G\restriction\alpha+1}$ is a symmetric extension of $\IS_\alpha^{G\restriction\alpha}$ by the symmetric system $\tup{\dot\QQ_\alpha^{G\restriction\alpha},\dot\sG_\alpha^{G\restriction\alpha},\dot\sF_\alpha^{G\restriction\alpha}}$ with the $\IS_\alpha^{G\restriction\alpha}$-generic filter $G(\alpha)$.\qed
\end{theorem}
\begin{theorem}
Suppose that $\tup{\QQ_\delta,\sG_\delta,\sF_\delta}\in\IS_\delta^G$ is a symmetric system and $H$ is a $V[G]$-generic filter for $\QQ_\delta$. Then there are names $\dot\QQ_\delta,\dot\sG_\delta$ and $\dot\sF_\delta$ in $\IS_\delta$, and a support $\vec K$ for all three of them, such that by shrinking $\dot\sG_\alpha$ to $\dot K_\alpha$ and restricting the filters $\dot\sF_\alpha$, we can obtain a symmetric iteration of length $\delta+1$, and $G\ast H$ is a $V$-generic filter for that iteration.\qed
\end{theorem}
Note that the when shrinking the $\dot\sG$'s, we might have to repeat the shrinking process, however the maximal coordinate where shrinking is necessarily will form a strictly decreasing sequence of ordinals, so it is guaranteed to stabilize in finitely many steps. It is also worth noting that this tells us that if we kept on going without shrinking the groups at all, then at the limit step we have a valid symmetric iteration, although the choice of generic filters might not generate a generic filter to this iteration.

One last thing to ponder about is why did we require that $G\ast H$ is $V$-generic for the iteration? Well. We did so because we wanted to ensure that we catch the antichains of $\dot\QQ_\delta^G$ which are not in $\IS_\delta^G$. But we might have had $\IS_\delta^G$-generic filters for $\dot\QQ_\delta^G$ which are not $V[G]$-generic themselves. What if we wanted to use one of those instead? The next section discusses a partial solution to this problem.

\section{The generics problem and productive iterations}\label{section:productive}
\subsection{Two-step motivation}\label{subsection:cohen example}
Let us start with a two-step iteration as a motivating example, our second step will involve no symmetries so we omit that part. The first step is going to be the standard Cohen model, like in \autoref{ex:cohen}. Force with $\PP=\Add(\omega,\omega)$, so a condition $\PP$ is a partial function $p\colon\omega\times\omega\to2$ with finite domain; $\sG$ is going to be the finitary permutations of $\omega$ where $\pi p(\pi n,m)=p(n,m)$ is the action on $\PP$; finally, $\sF$ is generated by groups of the form $\fix(E)=\{\pi\in\sG\mid\pi\restriction E=\id\}$ for $E\subseteq\omega$.

Let $\dot a_n$ denote the $\PP$-name for the $n$th Cohen real, i.e.\ $\dot a_n=\{\tup{p,\check m}\mid p(n,m)=1\}$, and let $\dot A$ denote the name $\{\dot a_n\mid n<\omega\}^\bullet$. Standard arguments show that $\dot A\in\HS_\sF$ and that $1_\PP\forces^\HS_\PP``\dot A$ is Dedekind-finite''.

Now we want $\dot\QQ$ to introduce a well-order of $\dot A$, specifically, of type $\omega$. The obvious way of doing so would be by considering finite and injective partial functions from $\dot A$ to $\omega$. But despite the fact that $\dot A$ itself will forget its enumeration in the symmetric extension, we can still utilize the fact that the name itself is countable in $V$ to give a very canonical description of a condition in $\dot\QQ$.

For a finite partial function $f\colon\omega\times\omega\to2$, so $f\in\PP$, let $\dot q_f$ be the name \[\{\tup{\tup{\dot a_n,\check m}^\bullet,\check f(n,m)}^\bullet\mid\tup{n,m}\in\dom f\}^\bullet.\]
Now define $\dot\QQ=\{\dot q_f\mid f\in\PP\}^\bullet$ and the order to be just reverse inclusion. It is not hard to check that $\dot\QQ$ is in fact hereditarily $\sF$-symmetric, and that $\PP\ast\dot\QQ$ does in fact satisfy the definition for a symmetric iteration of length $2$. Moreover, in the generic extension by $\PP$, it is very easy to see why $\dot\QQ$ is just isomorphic to $\PP$. But in the intermediate model there is an injection from $\dot A$ to $\dot\QQ$, whereas $\PP$ is countable, so the isomorphism itself is not symmetric.

One can show, next, that if $G$ is the $V$-generic filter used for $\PP$, then there is some $H\in V[G]$ such that $H$ is $\HS_\sF^G$-generic for $\dot\QQ^G$, and $\HS_\sF^G[H]=V[G]$. But of course, $G\ast H$ is not $V$-generic for $\PP\ast\dot\QQ$.

\subsection{So how do we access more generics?}
We want to find some condition on the iteration so we can use filters which are generic over the intermediate model we have at each step, rather than have genericity for the entire iteration over the ground model; and we often times it will be easy when the iteration is presented in a way where the iterands are ``almost isomorphic to forcings in the ground model''.

Our first wish is quite hard to satisfy. We used mixing in a very substantial way when defining the iterated automorphisms,\footnote{See \autoref{prop:aut-ext}.} and we saw in \autoref{section:justification} that mixing is also creating names which are themselves not stable under automorphisms directly (but they are eventually become equal to names which are stable enough). So just using filters which are not generic for the entire iterations will not be enough to interpret the names in $\IS$. The first place to look for a solution is Kunen iterations, where mixing was not involved in the definition of the forcing itself.

\subsection{Kunen iterations and partial automorphisms}\label{subsection:Kunen}
We defined the Kunen preordered set for an iteration $\PP\ast\dot\QQ$ as follows: the conditions are pairs $\tup{p,\dot q}$ where $p\forces_\PP\dot q\in\dot\QQ$, and $\dot q$ appears in $\dot\QQ$. Suppose that $\dot\sigma$ is a name such that $1_\PP\forces\dot\sigma\in\aut(\dot\QQ)$. In order to extend $\dot\sigma$ to $\gaut{\dot\sigma}$ we used mixing in a significant way, so now we can only apply $\gaut{\dot\sigma}$ to a condition $\tup{p,\dot q}$ if $p\forces_\PP\dot\sigma(\dot q)=\dot q'$ and $\dot q'$ appears in $\dot\QQ$ (note that this condition implies that $p\forces_\PP\dot q'\in\dot\QQ$, making $\tup{p,\dot q'}$ a condition in $\PP\ast\dot\QQ$).

So our automorphisms become partial automorphisms. Given $\gaut{\vec\pi}$ we can only apply it to a condition $p$ if each initial segment has enough information about the action on the remainder of the condition. Similarly we lose the ability to mix over supports, we have to resort only to excellent supports. Everything becomes far more complicated to state, as we always have to keep track what sort of information our condition already decides about a given automorphism sequence.

Of course, this is not a mathematical problem. We only ever used finitely many automorphisms in each of the proofs, so we can always make finitely many extensions to decide the needed information, and then argue via density and genericity.

But what good is a hammer which is too heavy to use? Ideally, we would have worked out this entire method and then we could have danced at both weddings: work with the Jech iterations for simplicity, but when you need to pick generics switch to Kunen iterations and argue your point. However the increased complexity compelled us to choose a different route. 
\subsection{The common ground: productive iterations}
In the case of $\PP\ast\dot\QQ$ as defined in \autoref{subsection:cohen example}, we had a very canonical description of the forcing. It was a description that allowed for $\PP\ast\dot\QQ$ to be presented as a Kunen iteration without $1_\PP$ losing any significant information about conditions of $\dot\QQ$ or their order. We distill this idea in the following definition.

\begin{definition}
A symmetric iteration $\tup{\dot\QQ_\alpha,\dot\sG_\alpha,\dot\sF_\alpha\mid\alpha<\delta}$ is called a \textit{productive iteration} if the following conditions hold:
\begin{enumerate}
\item For every $\alpha$, $\dot\QQ_\alpha,\dot\sG_\alpha$ and $\dot\sF_\alpha$ are $\bullet$-names.
\item The conditions of $\PP_\delta$ are exactly those where for every $\alpha<\delta$, $p(\alpha)$ appears in $\dot\QQ_\alpha$, i.e., we can think about $\PP_\alpha$ as a Kunen iteration.
\item For every $\alpha$, $1_\alpha$ decides the statement $\dot\pi_\alpha=\dot\sigma_\alpha$ for every two names which appear in $\dot\sG_\alpha$.
\item For every $\alpha$, for every $\dot q$ appearing in $\dot\QQ_\alpha$ and every $\dot\sigma$ appearing in $\dot\sG_\alpha$, there is some $\dot q'$ appearing in $\dot\QQ_\alpha$, such that $1_\alpha\forces_\alpha\dot\sigma(\dot q)=\dot q'$.
\end{enumerate}
We modify the construction of $\cG_\delta$ by only allowing automorphism sequences $\vec\pi$ that each $\dot\pi_\alpha$ appears in $\dot\sG_\alpha$. And we also modify the definition of $\cF_\delta$-respected names by requiring the support to be excellent with the property that $1_\alpha$ decides $\vec\pi\in\vec H$ for every $\vec\pi$ and $\vec H$.
\end{definition}
The motivation is that we have symmetric systems where the forcing itself is ``almost'' a ground model forcing, and we can use the ground model's automorphism group and filter of subgroups to induce in a canonical way an ``almost isomorphic'' structure in the intermediate extension.
\smallskip

We added new restrictions that limit our access to mixing, which was very essential for the definition. So we need to investigate what remains of the general theory. It might seem that every proof which has an underlying appeal for mixing (which is almost all of them) is going to fail now. However, every use of mixing can be circumvented with the additional assumptions. Obviously, \autoref{thm:IS-names are closed under mixing} fails in this context, but we retain the corollary that automorphism sequences are in $\IS$---as we restrict ourselves to automorphism sequences as in the definition of the productive iteration.

It is important to point out that if we have a productive iteration and we compute $\IS$ as in the general theory (with mixing and all), then every $\dot x\in\IS$ will be generically equal to a name which is in the ``productive definition'' of $\IS$. The proof is by induction on the rank of $\dot x$, and is very similar to the previous proofs by induction on the rank of a name. So despite the additional requirements, we can jump between the general theory and the productive theory when it comes to working with productive iterations.

\begin{definition}
Suppose that $\PP_\delta$ is a productive iteration.
\begin{enumerate}
\item $D\subseteq\PP_\delta$ is a \textit{symmetrically dense} if it is a dense set and there is some $\cF_\delta$-support $\vec H$ such that whenever $p\in D$ and $p\forces_\delta\vec\pi\in\vec H$, then $\gaut{\vec\pi}p\in D$.
\item $G\subseteq\PP_\delta$ is a \textit{$V$-symmetrically generic filter} if it is a filter and for every symmetrically dense open $D\in V$, $D\cap G\neq\varnothing$.
\end{enumerate}
\end{definition}
\begin{remark}
The above definition is general, but applying it to the productive case, we note that our supports are all excellent, and since $1_\delta$ decides all the elements of $\vec H$, it follows that if $p\forces_\delta\vec\pi\in\vec H$, then $1_\delta\forces_\delta\vec\pi\in\vec H$. So in effect the definition simply states that $D$ is closed under the automorphisms from $\vec H$.
\end{remark}
We would like to be able and choose our generics pointwise, but this may still cause problems at limit steps. Symmetrically dense sets eliminate this problem for limit steps, while very easily allowing us to choose generics if we only want to extend the iteration by finitely many steps. And we have a criterion for when such extensions can proceed through limit steps.
\begin{lemma}
Let $\varphi(x)$ be a statement in the language of forcing and $\dot x\in\IS_\delta$. Then the decision set $D=\{p\in\PP_\delta\mid p\forces^\IS_\delta\varphi(\dot x)\text{ or }p\forces^\IS_\delta\lnot\varphi(\dot x)\}$ contains a symmetrically dense open subset.
\end{lemma}
Of course, the lemma is formulated for a single-variable formulas for simplicity, but it holds for any formula.
\begin{proof}
Let $\vec H$ be an excellent support for $\dot x$, and fix a pre-dense set $E$ witnessing that $\dot x\in\IS$. If $p\forces_\delta^\IS\varphi(\dot x)$, $p\forces_\delta\vec\pi\in\vec H$, and there is some $q\in E$ such that $p\leq_\delta q$, then $p\forces_\delta\gaut{\vec\pi}\dot x=\gaut{\vec\pi^{-1}}\dot x=\dot x$. Therefore by the Symmetry Lemma, \[\gaut{\vec\pi}p\forces^\IS_\delta\varphi(\gaut{\vec\pi}\dot x)\land\gaut{\vec\pi}\dot x=\dot x.\] 

Therefore the set $\{p\in\PP_\delta\mid p\forces_\delta^\IS\varphi(\dot x)\lor p\forces^\IS_\delta\lnot\varphi(\dot x)\land\exists q\in E:p\leq_\delta q\}$ is a symmetrically dense open set.
\end{proof}
\begin{theorem}[The Productive Forcing Theorem for $\forces^\IS$]\label{thm:productive generics}
The following are equivalent for productive iterations:
\begin{enumerate}
\item $p\forces^\IS\varphi$.
\item For every $V$-symmetrically generic filter $G$ such that $p\in G$, $\IS^G\models\varphi$.
\item For every $V$-generic filter $G$ such that $p\in G$, $\IS^G\models\varphi$.
\end{enumerate}
\end{theorem}
\begin{proof}[Sketch of Proof]
The implication from (2) to (3) is trivial; and the implication from (3) to (1) follows from the fact that every condition lies in a $V$-generic filter. Finally, the proof from (1) to (2) is similar to the usual proof of the forcing theorem, here relying on the lemma above, ensuring that every decision set is met by the symmetrically generic filter.
\end{proof}
We can now give a very nice reformulation (and combination) of the theorems from the end of \autoref{section:justification}. As we saw, mixing was a great obstruction for nice factorization theorems. However, in the productive case, all the supports are excellent and all the automorphisms are very concretely decided. This means the following theorem is true.
\begin{theorem}
Let $\PP_\alpha$ be a productive iteration of length $\alpha$, and let $G_\alpha$ be a $V$-symmetrically generic filter. Suppose that $\tup{\QQ_\alpha,\sG_\alpha,\sF_\alpha}$ is a symmetric system in $\IS_\alpha^{G_\alpha}$ which has a name such that extending $\PP_\alpha$ by this name is still a productive iteration. Then $G_\alpha\ast G$ is $V$-symmetrically generic for $\PP_\alpha\ast\dot\QQ_\alpha$ if and only if $G$ is $\IS_\alpha^{G_\alpha}$-generic for $\QQ_\alpha$.\qed
\end{theorem}
\section{Preservation theorems}\label{section:preservation}
Recall that in this work we only consider finite support iterations. There are two major problems with finite support iterations and limit steps. The first and foremost is that at every limit step we add Cohen reals. Many Cohen reals. The second is that if we are not limiting ourselves to c.c.c.\ forcings, then we are likely to collapse cardinals, and many of them.

Both of these pose a problem if we want to look at anything remotely interesting. But we can mitigate our losses and show that at least assuming each step of the iteration is weakly homogeneous, we will only add reals when we explicitly add reals, and we will only collapse cardinals when we explicitly want to collapse cardinals. Not only this, but this extends, in fact, to class length iterations, which completely violate $\ZF$ in the class-generic extensions.

\begin{remark}
Note that we do not need to worry about productive or non-productive iterations. As a consequence of \autoref{thm:productive generics}, if we only prove these theorems using $V$-generic filters, the same will hold for symmetrically $V$-generic filters in the productive case.
\end{remark}
\subsection{Set forcing}
\begin{theorem}\label{thm:set preservation}
Suppose that $\PP_\delta$ is a symmetric iteration, and suppose that for every $\alpha<\delta$, \[1_\alpha\forces_\alpha^\IS\dot\sG_\alpha\text{ witnesses the homogeneity of }\dot\QQ_\alpha.\] If $\eta$ is such that there is $\alpha<\delta$, that for every $\beta>\alpha$, \[1_\beta\forces^\IS_\beta\tup{\dot\QQ_\beta,\dot\sG_\beta,\dot\sF_\beta}^\bullet\text{ does not add sets of von Neumann rank }<\check\eta,\] then whenever $\dot x\in\IS_\delta$ and $p\forces_\delta^\IS\rank(\dot x)<\check\eta$, there is some $q\leq_\delta p$ and some $\dot x'\in\IS_\alpha$ such that $q\forces_\delta\dot x=\dot x'$.
\end{theorem}
In other words, if $\delta$ is a limit ordinal, and $V_\eta$ (of $\IS_\alpha^G$'s) was stabilized before the $\delta$th step, then no new sets are added with von Neumann rank below $\eta$ at the limit step as well. So if no forcing added reals, the limit steps will not add reals either. At least under the assumption of homogeneity.
\begin{proof}
We prove this by induction on $\delta$. For $\delta=0$ there is nothing to prove, and for successor steps this is an easy consequence of the factorization theorems and the assumptions. Suppose that $\delta$ is a limit ordinal, we will prove the claim by induction on the rank of $\dot x\in\IS_\delta$.

Suppose that $\dot x\in\IS_\delta$, we define the name $\dot x'$ as follows: for every $\tup{p,\dot y}\in\dot x$ let $D(p,\dot y)$ be a maximal antichain below $p$ such that for every $q\in D(p,\dot y)$ there is some $\dot y_q\in\IS_\alpha$ such that $q\forces_\delta\dot y=\dot y_q$. Let $\dot x'=\{\tup{q,\dot y_q}\mid\exists\tup{p,\dot y}\in\dot x:q\in D(p,\dot y)\}$. It is not hard to see that $1_\delta\forces\dot x=\dot x'$, so we may assume that $\dot x=\dot x'$.

Let $\vec H$ be a support and $D$ be a pre-dense set witnessing that $\dot x\in\IS$, we may assume that every $p\in D$ is tenacious, $p\forces_\delta\vec H=\vec H_p$ for some excellent support, and $\vec H_p(p)$ fixes $p$. Let $p\in D$ and take $\beta$ such that $p\in\PP_\beta$ and $\beta>\max C(\vec H_p)$. We claim now that if $q\leq_\delta p$, and $q\forces_\delta\dot y\in\dot x$ for $\dot y$ appearing in $\dot x$, then $q\restriction\beta\forces_\beta\dot y\in\dot x$. Let us assume that we have shown that, now define \[\dot x_p=\{\tup{q\restriction\beta,\dot y}\mid q\leq_\delta p, q\forces_\delta\dot y\in\dot x\ \text{ and }\ \dot y\text{ appears in }\dot x\}.\] It follows that $p\forces_\delta\dot x_p=\dot x$. However every condition in $\dot x_p$ is in $\PP_\beta$, and every name appearing in $\dot x_p$ is in $\PP_\beta$. Therefore $\dot x_p$ is a $\PP_\beta$-name, moreover $\vec H_p\restriction\beta$ is an excellent $\cF_\beta$-support for $\dot x_p\restriction p$, so $\dot x_p\in\IS_\beta$. By the induction assumption there are names $\dot x_p'\in\IS_\alpha$ which $\dot x_p$ will be forced to be equal to them. But this equality goes up to $\delta$, as wanted.

So it remains to show that if $q\leq_\delta p$, and $q\forces_\delta\dot y\in\dot x$ for some $\dot y$ which appears in $\dot x$, then $q\restriction\beta\forces_\beta\dot y\in\dot x$. By the assumption, $\dot y$ is a $\PP_\alpha$-name, therefore any automorphism sequence $\vec\pi$ such that $\min C(\vec\pi)\geq\alpha$ must satisfy $\gaut{\vec\pi}\dot y=\dot y$, and moreover $p\forces_\delta\vec\pi\in\vec H_p$. Next, by homogeneity, any two extensions of $q\restriction\beta$ can be made compatible. Of course, if the two extensions are end-extensions, and $\vec\pi$ witnesses this, then we may assume that $\min C(\vec\pi)\geq\beta$, and so $q\restriction\beta\forces_\delta\gaut{\vec\pi}\dot x=\dot x$ and we are done. But now we are really done, otherwise there was some $q'$ such that $q'\leq_\beta q\restriction\beta$, with $q'\forces_\delta\lnot\dot y\in\dot x$, and this is impossible since $q'$ is compatible with $q$.

Therefore $q\restriction\beta\forces_\delta\dot y\in\dot x$, which allows us to define $\dot x_p$ and use the induction hypothesis on $\PP_\beta$ as wanted.
\end{proof}
\begin{remark}
Note that any continuous filtration of the universe can work in this context, not just the von Neumann hierarchy, and in fact the levels of the hierarchy need not be sets themselves. Namely, if $\{D_\alpha\mid\alpha\in\Ord\}$ is any continuous and uniformly definable filtration of the universe, and no new sets are added to $D_\eta$ above the stage $\alpha$ of the iteration, then under the assumption of homogeneity, $D_\eta$ is preserved at limit stages above $\alpha$ as well. The only necessary condition is that the $D_\alpha$'s have names stable under all automorphisms.
\end{remark}
\subsection{Class forcing}
\autoref{thm:set preservation} has an interesting corollary.

\begin{theorem}\label{thm:class preservation}
Suppose that $\tup{\dot\QQ_\alpha,\dot\sG_\alpha,\dot\sF_\alpha\mid\alpha\in\Ord}$ is a symmetric iteration and each step satisfies that $1_\alpha\forces_\alpha^\IS\dot\sG_\alpha\text{ witnesses the homogeneity of }\dot\QQ_\alpha$. Moreover, suppose that for every $\eta$, there is some $F(\eta)$ such that for every $\beta\geq F(\eta)$ no sets of von Neumann rank $<\eta$ are added by $\tup{\dot\QQ_\beta,\dot\sG_\beta,\dot\sF_\beta}$. Then whenever $G$ is a $V$-generic filter for $\PP=\PP_\Ord$, $\IS^G\models\ZF$.
\end{theorem}
\begin{proof}
Let $N$ be $\IS^G$. We cannot use the fact that $N$ is a transitive, almost universal class which is closed under G\"odel operations, since this route requires us to work inside a model of $\ZF$, and $V[G]$ fails to satisfy the power set axiom, as we added a proper class of Cohen reals (and we may have collapsed all the cardinals or even $\Ord$ itself to be countable). Instead, we verify all the axioms of $\ZF$ hold in $N$.

Extensionality, Infinity and Regularity are easy to verify, as neither is violated in $V[G]$, and $N$ is a transitive class of $V[G]$. Power set and Union hold as a corollary of \autoref{thm:set preservation}. If $x\in N$, then there is some $\eta$ such that $x\in V_\eta^N$, it follows that if $\dot x$ is a name for which $\dot x^G=\dot x$, we can assume that $\dot x\in\IS_{F(\eta)}$ and both the union and power set of $x$ will certainly be the same as they are computed in $\IS_{F(\eta+1)}$. It remains to check that Replacement holds in $N$.

Let $\varphi(x,y)$ be a formula and $A\in N$ (we omit the parameters for simplicity) such that $N\models\forall x\in A\exists!y\; \varphi(x,y)$. Fix $\dot A\in\IS$ such that $\dot A^G=A$. We say that $\alpha$ is a \textit{closure of $\PP$ for $\varphi$ and $\dot A$} if $V_\alpha\cap\PP=\PP_\alpha$, $\dot A\in V_\alpha\cap\IS_\alpha$, the formulas defining $p\forces^\IS\varphi(\dot x,\dot y)$ and $p\forces^\IS\forall x\in\dot A\exists!y\;\varphi(x,y)$ are reflected in $V_\alpha$, and moreover there exists some $p\in G\cap V_\alpha$ such that $p\forces^\IS\forall x\in\dot A\exists!y\;\varphi(x,y)$. We can find such closure point by the reflection theorem.

Let $\alpha$ be a closure point for $\varphi$ and $\dot A$, and let $p$ be a condition witnessing that $\varphi$ defines a function on $A$. Using mixing in $\PP_\alpha$, let $\{\dot x_i\mid i\in I\}$ be such that $p\forces_\alpha^\IS\dot A=\{\dot x_i\mid i\in I\}^\bullet$ and let $\dot y_i\in\IS_\alpha\cap V_\alpha$ such that $p\forces^\IS_\alpha\varphi(\dot x_i,\dot y_i)$ for all $i\in I$. We claim that $\dot B=\{\dot y_i\mid i\in I\}^\bullet$ is the name we are looking for to witness this instance of Replacement holds in $N$. By the choice of $\alpha$ as a closure point, $p\forces^\IS\varphi(\dot x_i,\dot y_i)$ for all $i\in I$. Now suppose that $q\leq p$, then $q\forces^\IS\forall x\in\dot A\exists!y\;\varphi(x,y)$, and if $q\forces^\IS\varphi(\dot x_i,\dot y)\land\varphi(\dot x_i,\dot y_i)$, then it is the case that $q\forces^\IS\dot y=\dot y_i$. 

Therefore $\dot B^G\in N$ is exactly the set $\{y\mid\exists x\in A\;\varphi(x,y)\}$, as wanted.
\end{proof}
\section{Example: failure of Kinna--Wagner Principles}\label{section:kwp}
\subsection{On Kinna--Wagner Principles}
The Kinna--Wagner Principle states that every set can be injected into the power set of an ordinal. It is a weakening of the axiom of choice that was given by Willy Kinna and Klaus Wagner in \cite{KinnaWagner:1955}. The principle implies that every set can be linearly ordered, however it is independent from statements like the Boolean Prime Ideal Theorem over $\ZF$ (see \cite{Pincus:1997} for a discussion on the subject). 

Monro defined, for a natural number $n$, $\KWP_n$ to be the statement that for every $X$ there exists an ordinal $\eta$, such that $X$ can be injected to $\power^n(\eta)$. In \cite{Monro:1973}, Monro extends the classic theorem of Vop\v{e}nka and showed that if $M$ and $N$ are transitive models of $\ZF$ with the same $(n+1)$th power set of ordinals, and $M\models\KWP_n$, then $M=N$. We will extend both Monro's definition and his extension of the Balcar--Vop\v{e}nka theorem below.
\begin{definition}
We define the \textit{$\alpha$-set of ordinals} hierarchy by recursion. $0$-sets ordinals are sets of ordinals; $\alpha$-sets of ordinals are sets of $\beta$-sets of ordinals for $\beta<\alpha$. Namely, an $\alpha$-set of ordinals is a subset of $\power^\alpha(\eta)$ for some ordinal $\eta$.\end{definition}
We will generally shorten the term, and just write that $A$ is an $\alpha$-set. Note that by recursion and usual coding arguments, $\alpha$-sets are closed under unions, finite products and even finite sequences. Namely, if $A$ is an $\alpha$-set, then $A^{<\omega}$ and $[A]^{<\omega}$ can be encoded as an $\alpha$-set as well.
\begin{definition}
For an ordinal $\alpha$, $\KWP_\alpha$ is the statement that every set $A$ is equipotent with an $\alpha$-set. We will abbreviate the statement $\exists\alpha\;\KWP_\alpha$ simply to $\KWP$.
\end{definition}
The axiom of choice is $\KWP_0$ and the classic Kinna--Wagner principle is $\KWP_1$. But as Monro proved in \cite{Monro:1973}, for every $n$, $\KWP_{n+1}\nrightarrow\KWP_n$. 
\begin{theorem}[The Generalized Balcar--Vop\v{e}nka--Monro Theorem]\label{thm:gen-bvm-thm}
Let $M$ and $N$ be two transitive models of $\ZF$ with the same $\alpha$-sets. If $M\models\KWP_\alpha$, then $M=N$.
\end{theorem}
The proof of the Balcar--Vop\v{e}nka theorem can be found in \cite[Theorem~13.28]{Jech:ST}, and it generalizes quite immediately here.
\begin{proof}[Sketch of Proof]
Every set in $M$ can be encoded as a well-founded relation on an $\alpha$-set. By the Mostowski collapse lemma, every set in $M$ lies in $N$. The direction $N\subseteq M$ is proved by induction on the von Neumann rank of members of $N$.
\end{proof}
\medskip

Andreas Blass in \cite{Blass:1979} defined $\SVC(X)$ to be the statement ``For every $A$ there exists an ordinal $\eta$ such that there is a surjection from $X\times\eta$ onto $A$'', let $\SVC$ be the statement $\exists X\;\SVC(X)$. We call such $X$ a \textit{seed}. Blass also proved in that same paper, that if $M$ is a symmetric extension of $V$, where $V\models\ZFC$, then $M$ satisfies $\SVC$; and that $M\models\SVC$ if and only if there is a generic extension of $M$ satisfying the axiom of choice.
\begin{theorem}
Suppose that $M\models\SVC$. Then $M\models\KWP$. Consequently, every symmetric extension of a model of $\ZFC$ satisfies $\KWP$.
\end{theorem}
\begin{proof}
If $M\models\SVC$, let $X$ be a witness of that, and let $\alpha$ be the von Neumann rank of $X$. Then $X$ is an $\omega+\alpha$-set, and by induction on $\alpha$, it is equipotent with an $\alpha$-set. It follows that every set can be injected into $\power(X\times\eta)$ for some ordinal $\eta$. Therefore $\KWP_{\alpha+1}$ holds.
\end{proof}

\subsection{Failure of Kinna--Wagner Principles}
To prove that $\KWP_{n+1}\nrightarrow\KWP_n$, Monro constructed a sequence of transitive models, $M_n$, such that $M_n$ and $M_{n+1}$ have the same $n$-sets, but not the same $n+1$-sets, and $\KWP_n$ fails to hold in $M_{n+1}$. The idea was to add at each step a Dedekind-finite set $A_n$, and then construct a model similar to Cohen's model, only adding $\omega$ subsets to $A_n$. Monro used relative constructibility instead of symmetric extensions, but we know from the work of Grigorieff in \cite{Grigorieff:1975} that the two methods are tightly connected. The limitation of Monro's method was that it had to stop at $\omega$: the increasing union of models of $\ZF$ is not necessarily a model of $\ZF$ again.
\smallskip

We present here a reconstruction of Monro's work within the framework of iterated symmetric extensions. This lets us coalesce these models into one, where $\KWP_n$ fails for all $n<\omega$ at the same time. We will discuss the difficulties in extending this construction beyond $\omega$ steps. 

\begin{theorem}
For all $n<\omega$ there is a model where $\KWP_{n+1}$ holds, but $\KWP_n$. Moreover, there is a model of $\ZF$ where $\KWP_n$ fails for all $n<\omega$ simultaneously.
\end{theorem}
\begin{proof}
We define for all $n$ a symmetric system which is weakly homogeneous and does not add $n$-sets, and that the $n$th iteration forces (relative to $\IS_n$) that $\KWP_n$ holds. The idea is to iterate the construction of the Cohen model---like in Monro's original paper---where we start by adding $\omega$ Cohen reals, then add $\omega$ Cohen subsets to the set of Cohen reals, then $\omega$ subsets to the set obtained and so on. This approach fails when reaching limit steps. We will discuss this failure, and what can be done to overcome it at the end of the proof.

We begin by giving a rigorous definition of the symmetric system we will use, it is weakly homogeneous by its finitary nature. Unfortunately, it does not simplify the definition to consider this specific context, so we give a seemingly convoluted definition for a general context. The first one is the way we make a generic copy of adding $\omega$ subsets to a given name $\dot X$, indexed by some set $I$.

We find it clearer to separate $I$ and $\omega$ in this definition as it is more informative as to which role each index is taking, even though for our case $I$ will always be $\omega$ itself. The definition is based on the following principle, if we add $|I|$ Cohen reals, then we have added $\omega$ subsets to $I$, by simply rotating the matrix of the Cohen reals. 

\begin{definition}
Suppose that $\dot X=\{\dot x_i\mid i\in I\}^\bullet$ is a name such that for all $i\neq j$, $1\forces\dot x_i\neq\dot x_j$. For $f\in\Add(\omega,I)$ we let $\dot q_f=\{\tup{\tup{\check n,\dot x_i}^\bullet,f(i,n)}^\bullet\mid\tup{i,n}\in\dom f\}^\bullet$, and let $\Add^\bullet(\dot X,\check\omega)$ denote the name $\{\dot q_f\mid f\in\Add(\omega,I)\}^\bullet$. Let $S_\omega$ denote the group of finitary permutations of $\omega$, then for any $\pi\in S_\omega$, there is a natural action of $\pi$ on $\Add(\omega,I)$ given by $\pi f(i,\pi n)=f(i,n)$ and this translates to a natural action on $\Add^\bullet(\dot X,\omega)$ given by $\dot\pi\dot q_f=\dot q_{\pi f}$. 

The \textit{standard Cohen system for $\dot X$} is the symmetric system given by $\Add^\bullet(\dot X,\check\omega)$ with the permutation group induced by $S_\omega$ and the normal filter of subgroups given by finite stabilizers. The set of \textit{canonical generics} for $\Add^\bullet(\dot X,\check\omega)$ is given by the name \[\{\tup{1_A,\dot X_n}^\bullet\mid n<\omega\}^\bullet,\] where $\dot X_n=\{\tup{\dot q_f,\dot x_i}^\bullet\mid f(i,n)=1\}^\bullet$. These names are in fact names-of-names, and are analogous to the canonical Cohen reals and the canonical name for the set of Cohen reals in the extension where $\dot X$ is interpreted.
\end{definition}
It is not hard to see that whenever $\dot X\in\IS$ for some symmetric iteration satisfies the condition of the definition, then the standard Cohen system for $\dot X$ is also in $\IS$ and has the same support as $\dot X$. In particular, if all automorphisms respect $\dot X$, then the standard Cohen system is a candidate for continuing the iteration.

If $p$ is a condition in the standard Cohen system and $E$ is a finite subset of $\omega$, we write $p\restriction E$ to mean $p\restriction(E\times\omega)$. Namely, $p\restriction E$ is the restriction of $p$ to only the canonical generics with indices in $E$. As we are going to iterate these forcings in the proof that will now follows, if $\vec E$ is a sequence of finite sets of $\omega$ of the length of the iteration, and $p$ is a condition in the iteration, we will write $p\restriction\vec E$ to denote the pointwise restriction $p(n)\restriction E_n$.

Our assumptions on the ground model are simply $\ZFC$. We will refer to some proofs written elsewhere under assumption of $V=L$, or the existence of some global well-ordering, however these are only cosmetic assumptions for simplifying these proofs. We define by induction our symmetric systems. Suppose that the symmetric iteration $\tup{\PP_n,\cG_n,\cF_n}$ was defined, and $1_n\forces_n^\IS\KWP_n\land\forall k<n,\lnot\KWP_k$.

For $n=0$, let $\tup{\QQ_0,\sG_0,\sF_0}$ be the standard Cohen system taking $I=\omega$. This is exactly the standard Cohen model. We denote by $\dot a_{1,m}$ the $m$th canonical generic of the system and by $\dot A_1$ the name for the set of canonical generics. It is a standard fact that $1_1\forces^\IS_1\KWP_1$.\footnote{It is wise to pause and contemplate the path we slowly walked in life, if we ended up with a notation as odd as $1_1$.}

Suppose that the $\PP_n$ was defined. Let $\tup{\dot\QQ_n,\dot\sG_n,\dot\sF_n}^\bullet$ be the standard Cohen system for $\dot A_n$. We denote by $\dot a_{n+1,m}$ the $m$th canonical generic, and $\dot A_{n+1}$ is the set of canonical generics, $\{\dot a_{n+1,m}\mid m<\omega\}^\bullet$. Using a result of Halpern--Levy (see \cite[~Ch.~5,~Problem~23]{Jech:AC}, originally in \cite{HalpernLevy:1967}). the partial interpretation function needed for \cite[~Lemma~5.25]{Jech:AC} to go through is definable in $\IS_{n+1}$. The lemma provides a definable injection from the Cohen model into $I\times\Ord$, where $I$ is the set of finite subsets of $A_1$ above. As the assumptions in Jech include $V=L$, this is indeed the result provable, but in the general setting one can replace $\Ord$ by a sufficiently large ordinal $\eta$, and have the result localized to each $V_\eta$.

In our setting we replace $\Ord$, or even $\eta$, by $\power^n(\Ord)$ or $\power^n(\eta)$ for a large enough $\eta$; and $I$ by the finite subsets of $\dot A_{n+1}$ which is an $n+1$-set. And so the result is that if $\dot x\in\IS_{n+1}$, then $\forces^\IS_{n+1}``\exists\dot y(|\dot x|=|\dot y|)$ and $\dot y$ is an $n+1$-set''. In other words, $1_{n+1}\forces^\IS_{n+1}\KWP_{n+1}$.

Our proof would be complete if we can prove that moving from $\IS_n$ to $\IS_{n+1}$ we do not add new $k$-sets for $k<n$. In that case, as a consequence of \autoref{thm:gen-bvm-thm} each of the models has to be distinct, and since an element of $V_{\omega+n+1}$ is always coded by an $n$-set, this would imply the stabilization needed for the preservation theorem, \autoref{thm:set preservation}. Moreover, as we remarked at the end of \autoref{section:preservation}, we can talk about any uniformly definable filtration of the universe, not just the von Neumann ranks. This applies in this case to the hierarchy of $\alpha$-sets. In particular, the limit of the iterations will not satisfy $\KWP_n$ for any $n$, since its $n$-sets came from the $n+1$th step of the iteration, where $\KWP_n$ fails.

We shall prove, by induction on $n$, that if $\dot x\in\IS_n$ is a name of a $k$-set for $k<n$, then there is $p\in\PP_n$ and some $\dot y\in\IS_k$ such that $p\forces^\IS_n\dot x=\dot y$. Or in other words, that no new $k$ sets were added when forcing with the $n$th symmetric system.

For $n=0$, there is nothing to check (morally this would amount to checking that no ordinals were added, which is of course the case with forcing). Suppose that $n=m+1$. We will prove the claim by induction on $k<n$. Suppose that $\dot x\in\IS_n=\IS_{m+1}$ is such that $1_n\forces^\IS_n``\dot x\text{ is a }\check k\text{-set}$'' for some $k<n$, by standard arguments we may assume that every name which appears in $\dot x$ is in fact in $\IS_k$. Let $\vec E$ be a sequence of finite subsets of $\omega$ such that $\tup{\fix(E_k)\mid k<n}$ is a support for $\dot x$.

By standard homogeneity arguments we have that if $p\forces^\IS_n\dot u\in\dot x$ and $\dot u\in\IS_k$, then $p\restriction\vec E\forces^\IS_n\dot u\in\dot x$. If we can show that there is a fixed condition $r$ such that $\supp(r)\subseteq(j,n)$ such that if $p\forces^\IS_n\dot u\in\dot x$, then $(p\restriction k)^\smallfrown r\forces^\IS_n\dot u\in\dot x$, then it will be enough to conclude that $\dot x$ can be reduced to a name in $\IS_j$, and if $j\leq m$, the induction hypothesis is enough to take us all the way to $k$ itself. So in fact it is enough to simply get rid of the $m$th coordinate in this fashion. To do that, we will need the following lemma.

\begin{lemma}\label{lemma:kwp-up-hom}
For all $n$, given $p\in\PP_n$ and a condition $\dot q$ in $\dot\QQ_n$, there is $p'\leq_n p$ and finite set $D$ which depends only on $p'$ such that if $p'\forces_n\dot q'\leq_{\QQ_n}\dot q\restriction\omega\times D$, then there is $\vec\pi\in\cG_n$ such that $\gaut{\vec\pi}p'=p'$ and $p'\forces_n``\gaut{\vec\pi}\dot q'$ is compatible with $\dot q$''.
\end{lemma}

Assume the lemma. Without loss of generality, $p=p'$ as in the lemma, and that $p(m)$ having the property that its domain is maximal with respect to the finite set guaranteed by the lemma. Now by homogeneity we get that given any name $\dot u$ for a $k'$-set for $k'<k$, we can decide $\dot u\in\dot x$ by only extending $p\restriction m$. This finished the proof indeed, as wanted.
\end{proof}
\begin{proof}[Proof of \autoref{lemma:kwp-up-hom}]
For the case $n=0$, $\PP_0$ is the trivial forcing and there is nothing to prove. Suppose that $n=m+1$, in that case $\dot\QQ_n$ is a standard Cohen system and $\PP_n$ is an iteration of Cohen systems too. Let $p$ be a condition and $\dot q$ be a name for a condition in $\dot\QQ_n$. First extend $p$ to $p'$ such that:
\begin{enumerate}
\item $p'$ has the property that $p\restriction k$ decides the condition $f$ such that $p(k)=\dot q_f$. In particular the domains of each coordinate of $p$ are decided.
\item $p'$ itself decides that $\dot q=\dot q_f$ for some $f\in\Add(\omega,\omega)$.
\end{enumerate}
For readability, assume that $p=p'$. Let $D$ be the union of all the domains of $p(i)$ for $i<n$, then it is finite and we claim that $D$ is the wanted set. Suppose that $\dot q'$ is such that $p\forces_n\dot q'\leq_{\QQ_n}\dot q\restriction\omega\times D$. Assume momentarily that $p$ also decides that $\dot q'=\dot q_g$ for some $g\in\Add(\omega,\omega)$.

We can now find a finitary permutation of $\omega$, $\pi$ such that $\pi$ acts on $\dot\QQ_m$, and thus on $\PP_n$, with the following properties:
\begin{enumerate}
\item $\pi\restriction D=\id$.
\item If we consider $g'(i,j)=g(i,\pi j)$, then $g'$ and $f$ are compatible.
\end{enumerate}
Once we find such $\pi$, let it be $\pi_m$ such that $\pi_j=\id$ for $j<m$ in the sequence $\vec\pi$. Then $\gaut{\vec\pi} p=p$ by the choice of $D$, and $\gaut{\vec\pi}\dot q_g=\dot q_{g'}$ which is compatible with $\dot q\restriction\omega\times D$.

Of course, such $\pi$ is easy to find. Simply ``move'' anything outside of $D$ to a part disjoint from the domain of $f$. If $p$ does not decide the value of $\dot q'$ as $\dot q_g$, then find a maximal antichain below it which does, for each one of these conditions find a suitable automorphism as above, and use mixing to define $\pi$.
\end{proof}

\subsection{Transfinite failures} So what happens when we try to go forward with the construction? We have an obvious $\omega$-set of ordinals in the form of $A_\omega=\bigcup_{n<\omega}A_n$, and we would like to add a new subset to $A_\omega$. The obvious thing to do is to add it with finite conditions, but this is a problem, since you want to ensure that no $n$-sets are added, and a finite conditions approach will add a new subset to each of the $A_n$'s. So instead we need to choose the intersection of the new subset with each $A_n$. But this too leads us to some problems.

If we choose $B_n\subseteq A_n$, then just by asking what are the $n$'s for which $B_n$ is finite, or co-finite, or co-infinite, we can already code a real number. We can ask whether or not $B_n$ is a subset of those elements in $B_{n+1}$, and so on. So approximating using finite conditions is surely to introduce a new real. And we only considering adding a single subset.

Assaf Shani solves this problem in \cite{Shani:unpub} by requiring that we choose one element from each $A_n$, and that $a_n\in a_{n+1}$ for all $n$. So we create an increasing $\in$-chain. This argument works for adding a single choice function, however when adding two generic choice sequences, the coordinates on which they agree is already going to be a new real, so the approach does not solve the whole problem. Instead, Shani solves this by introducing a generic tree, such that adding any finitely many branches to the tree will not add new reals, or any $n$-sets for $n<\omega$. This approach can be generalized to any countable ordinal length by adding a cofinal sequence of tree-like forcings which are ``sufficiently distant'' from one another. However, when trying to push this approach to $\omega_1$, Fodor's lemma shows there is no way to add these forcings in a coherent way. 

We finish by pointing out that while there is no proof that $\KWP_{\alpha+1}\nrightarrow\KWP_\alpha$ for \textit{all} $\alpha$, in \cite{Karagila:Bristol} the author construct a sequence of models $M_\alpha$, and their union $M$, such that $M$ has the same $\alpha$-sets as $M_\alpha$. In particular, $\KWP$ fails in $M$ and $\KWP_\alpha$ fails in $M_\alpha$. It is unclear, however, whether or not $\KWP_{\alpha+1}$ holds in $M_\alpha$ or how does the failures ``move up'' between the different models. Therefore the only thing we can confidently say is that there is a proper class of ordinals such that if $\alpha<\beta$ are both in the class, then $\KWP_\beta$ does not imply $\KWP_\alpha$ in $\ZF$.

\section{Coda}\label{section:coda}

As with any new mathematical technology, we hope that there will be some exciting uses for this technology in the future. Some of these include clearer and easier reconstruction of Gitik's model and of those of Sageev. In \cite{Karagila:Bristol} we use a class-length productive iteration to construct an intermediate model of $\lnot\KWP$ which lies between $L$ and Cohen real extension $L[c]$.

\begin{question}Is iterating symmetric extensions the same as a single symmetric extension?
\end{question}
We know that in the context of forcing the answer is yes: iterated generic extensions can be presented as a single generic extension. And while a positive answer to the above question might seem to nullify some of the efforts taken here, it is still a different approach to constructing models of $\ZF$, and we can still reach class-length iterations under some hypotheses. Using Grigorieff's work, we know that symmetric extensions are models of the form $\HOD(V\cup x)^{V[G]}$ for some $x\in V[G]$. The problem is that $\HOD$-type models are not robust between models, so there is no guarantee that a symmetric extension of a symmetric extension is itself necessarily a symmetric extension, let alone a limit of symmetric extensions.

\begin{question}
Can the construction be extended in a reasonable way to any other type of iterations (e.g.\ $\kappa$-support iteration)?
\end{question}

If the first question is answered positively, this might shed some light on the second answer.

It is often the case that a forcing over a model of $\ZF$ will well-order some part of the universe. For example, when $\kappa$ is regular, $\Add(\kappa,\kappa)$ introduces a bijection between $\kappa^{<\kappa}$ and $\kappa$. This means that iterating Cohen-like constructions one step at a time will only violate choice at limit steps. With finite support iterations, Dependent Choice is sure to fail at those limit steps. However, with countable support iterations, we might salvage Dependent Choice, or other weak choice principles.

\subsection{Products of symmetric systems}
The absence of \textit{products} of symmetric system from this work so far is quite notable. Much like the fact that products can be viewed as iterations where all forcing notions come from the ground model, we can also look at a product of symmetric systems. There, we have $\tup{\QQ_\alpha,\sG_\alpha,\sF_\alpha}$ for some $\alpha<\delta$, we define a some kind of product $\PP_\delta$ (finite support, countable support, Easton, etc.) of the forcing notions. We can define a product of the groups with a similar support (or a different one, if we choose to) as a group of automorphisms of the product $\PP_\delta$ acting pointwise on each component. Then we can define a product of the filters, again with whatever support that our heart desires. Therefore the product of symmetric system is itself a symmetric system. But it is also an iteration. Moreover, since the iteration can be presented using only canonical $\check x$-style names, it is a productive one.

Indeed, this gives us a small step towards answering both of the questions we raised. Yes, the product of symmetric systems is a symmetric system. Yes, we can extend it to a variety of supports, not just finite support, granted that we only use ground model objects. But it also points out obvious difficulties, since in general a countable support product is not a countable support iteration.

\section*{Acknowledgments}
As with any work that is painstakingly developed over the course of a long time, there are many people to thank, more than it would be appropriate to recount here. Three people were particularly helpful to the overall process: Menachem Magidor, my advisor, for his constant help and advice; Yair Hayut, my colleague and friend, for his patience and many helpful ideas; and Martin Goldstern who visited in Jerusalem during the spring of 2015 and was willing to listen the then-current drafts and identify weak points. Without the three of them, it is unclear if this work would have been finished as quickly as it has. The anonymous referee also deserves their fair share of praise for reading this paper, start to finish, and making helpful remarks and suggestions.

Additionally, it is uncommon, but the brave readers who read this far deserve a mention for their tenacity: thank you!\footnote{Assuming the reader did not cheat and jump directly to this part of the paper, in which case we reserve this acknowledgement for when the reader actually reads the paper.}
\bibliographystyle{amsplain}
\providecommand{\bysame}{\leavevmode\hbox to3em{\hrulefill}\thinspace}
\providecommand{\MR}{\relax\ifhmode\unskip\space\fi MR }
\providecommand{\MRhref}[2]{%
  \href{http://www.ams.org/mathscinet-getitem?mr=#1}{#2}
}
\providecommand{\href}[2]{#2}

\newpage
\appendix
\section{More on tenacity}\label{section:appendix}
\begin{definition}
We say that two symmetric systems $\tup{\PP,\sG,\sF}$ and $\tup{\PP',\sG',\sF'}$ are \textit{equivalent} if for every $V$-generic $G\subseteq\PP$ there is a filter $G'\subseteq\PP'$ such that $\HS_\sF^G=\HS_{\sF'}^{G'}$ and vice versa.\footnote{Note that $G'$ need not be $V$-generic, just to interpret $\HS$-names correctly.}
\end{definition}
The following theorem is joint work with Yair Hayut, written here with his permission.
\begin{theorem}
Every symmetric system is equivalent to a tenacious system.
\end{theorem}
\begin{proof}
Let $\tup{\PP,\sG,\sF}$ be a symmetric system. We may assume without loss of generality that $\PP$ is a complete Boolean algebra. Let $\BB$ be the subalgebra of tenacious conditions, namely $p\in\BB$ if and only if there is some $H\in\cF$ such that for all $\pi\in\sG$, $\pi p=p$. Due to the fact that $\sF$ is normal, we get that $\BB$ is closed under $\sG$. Namely, if $p\in\BB$ is fixed by $H$, and $\pi\in\sG$, then $\pi p$ is fixed by $\pi H\pi^{-1}$. Therefore $\tup{\BB,\sG,\sF}$ is a symmetric system. Let us denote by $\HS_\PP$ and $\HS_\BB$ the two classes of hereditarily symmetric names.

It is clear that $\HS_\BB\subseteq\HS_\PP$. Let us show that if $\dot x\in\HS_\PP$, then there is some $\dot y\in\HS_\BB$ such that $1_\PP\forces_\PP\dot x=\dot y$. We prove this by induction on the rank of $\dot x$, so we may assume that every $\dot u$ appearing in $\dot x$ is in fact in $\HS_\BB$. If $\tup{p,\dot u}\in\dot x$, let $\bar p$ denote $\sum\{\pi p\mid\pi\in\sym(\dot x)\cap\sym(\dot u)\}$. This condition is well-defined since $\PP$ is a complete Boolean algebra. Moreover, by the Symmetry Lemma we get that $p\forces_\PP\dot u\in\dot x$ if and only if $\pi p\forces_\PP\dot u\in\dot x$ for $\pi\in\sym(\dot x)\cap\sym(\dot u)$ and therefore $\bar p\forces_\PP\dot u\in\dot x$ as well. Finally, if $\pi\in\sym(\dot x)\cap\sym(\dot u)$, then $\pi\bar p=\bar p$ by the very definition of $\bar p$, and therefore $\bar p\in\BB$. Now define $\dot y=\{\tup{\bar p,\dot u}\mid\tup{p,\dot u}\in\dot x\}$, then it is easy to see that the two names are equivalent and that $\dot y\in\HS_\BB$.

It remains to show that the two are equivalent. Clearly, if $G\subseteq\PP$ is $V$-generic, then it is also generic for $\BB$. So we only need to show that if $G\subseteq\BB$ is $V$-generic, then in $V[G]$ there is a sufficiently generic filter $G'\subseteq\PP$ such that $\HS_\PP^{G'}=\HS_\BB^G$. Let $G'$ simply be the upwards closure of $G$ in $\PP$.

First, $G'$ is a filter, if $p,p'\in G'$ then there are some $\bar p,\bar p'\in G$ such that $\bar p\leq_\PP p$ and $\bar p'\leq_\PP p'$. As $G$ is a filter, there is some $r\in G$ such that $r\leq_\BB\bar p,\bar p'$, and therefore $r\in G$ and $r\leq_\PP p,p'$. By its definition, $G'$ is upwards closed. So it is indeed a filter.

Next, we will show that if $\dot x\in\HS_\PP$, then the interpretation of $\dot x$ by $G'$ is the same as $\dot y^G$, where $\dot y$ is the name defined above. Again, this will be done by induction on the rank of $\dot x$, using the definition of $\bar p$ as $\sum\{\pi p\mid\pi\in\sym(\dot x)\cap\sym(\dot u)\}$ is a weaker condition than $p$. Recall that
\[\dot x^{G'}=\{\dot u^{G'}\mid\exists p\in G':\tup{p,\dot u}\in\dot x\},\]
Let $\tup{p,\dot u}\in\dot x$, there is some $\dot v\in\HS_\BB$ such that $1_\PP\forces_\PP\dot u=\dot v$, and $\tup{\bar p,\dot v}\in\dot y$. By the induction hypothesis $\dot u^{G'}=\dot v^G$, but since $\bar p\in G$ and $p\in G'$ we get that indeed $\dot x^{G'}=\dot y^G$ as wanted.
\end{proof}
\end{document}